\RequirePackage{fix-cm}
\documentclass{svjour3}          
\smartqed  
\usepackage{graphicx}
\usepackage{amsmath,amsfonts, amssymb}
\usepackage{url}
\usepackage{cite}
\newcommand{\R}{\mathbb{R}}
\newcommand{\M}{\mathcal{M}}
\newcommand{\N}{\mathbb{N}}
\newcommand{\schild}{\textrm{schild}}
\newcommand{\fs}{\textrm{fs}}
\newcommand{\pole}{\textrm{pole}}
\newcommand{\se}{\mathfrak{se}}
\begin{document}

\title{Numerical Accuracy of Ladder Schemes for Parallel Transport on Manifolds}

\titlerunning{Ladder Schemes}        

\author{Nicolas Guigui         \and
       		Xavier Pennec 
}


\institute{N. Guigui \at
              Universit\'e C\^ote d'Azur and Inria, Epione team \\
              \email{nicolas.guigui@inria.fr}           
}

\date{Submitted \today}

\maketitle

\begin{abstract}
Parallel transport is a fundamental tool to perform statistics on Riemannian manifolds. Since closed formulae don't exist in general, practitioners often have to resort to numerical schemes.
\textit{Ladder} methods are a popular class of algorithms that rely on iterative constructions of geodesic parallelograms. And yet, the literature lacks a clear analysis of their convergence performance.
In this work, we give Taylor approximations of the elementary constructions of Schild's ladder and the pole ladder with respect to the Riemann curvature of the underlying space. We then prove that these methods can be iterated to converge with quadratic speed, even when geodesics are approximated by numerical schemes. 
We also contribute a new link between Schild's ladder and the Fanning Scheme which explains why the latter  naturally  converges only linearly. 
The extra computational cost of ladder methods is thus easily compensated by a drastic reduction of the number of steps needed to achieve the requested accuracy.
Illustrations on the 2-sphere, the space of symmetric positive definite matrices and the special Euclidean group show that the theoretical errors we have established are measured with a high accuracy in practice. 
The special Euclidean group with an anisotropic left-invariant metric is of particular interest as it is a tractable example of a non-symmetric space in general, which reduces to a Riemannian symmetric space in a particular case.
As a secondary contribution, we compute the covariant derivative of the curvature in this space. 
\end{abstract}

\section{Introduction}
\label{intro}
In many applications, it is natural to model data as points that lie on a manifold. Consequently, there has been a growing interest in defining a consistent framework to perform statistics and machine learning on manifolds \cite{pennec_riemannian_2020}. 
A fruitful approach is to locally linearize the data by associating to each point a tangent vector. The parallel transport of tangent vectors then appears as a natural tool to compare tangent vectors across tangent spaces.
For example, \cite{brooks_riemannian_2019} use it on the manifold of symmetric positive definite (SPD) matrices to centralize batches when training a neural network. In \cite{yair_parallel_2019}, parallel transport is used again on SPD matrices for domain adaptation. In \cite{kim_smoothing_2019}, it is a key ingredient to spline-fitting in a Kendall shape space. In computational anatomy, it allows to compare longitudinal, intra-subject evolution across populations \cite{lorenzi_IPMI_2011, lorenzi_efficient_2014, cury_spatio-temporal_2016, schiratti_bayesian_2017}. It is also used in computer vision in e.g. \cite{hauberg_unscented_2013, freifeld_model_2014}.

However, there is usually no closed-form solution to compute parallel transport and one must use approximation schemes. Two classes of approximations have been developed. The first, intends to solve ordinary or partial differential equations (ODE/PDE) related to the definition of parallel transport itself or to the Jacobi fields \cite{younes_jacobi_2007, louis_fanning_2018}. For instance, \cite{kim_smoothing_2019} derived a homogeneous first-order differential equation stemming from the structure of quotient space that defines Kendall shape spaces. \cite{louis_fanning_2018} leveraged the relation between Jacobi fields and parallel transport to derive a numerical scheme that amounts to integrating the geodesic equations. They prove that a convergence speed of order one is reached. This scheme is particularly appealing as only the Hamiltonian of the metric is required, and both the main geodesic and the parallel transport are computed simultaneously when no-closed form solution is available for the geodesics.

The second class of approximations, referred to as \textit{ladder} methods \cite{lorenzi_IPMI_2011, lorenzi_efficient_2014}, consists in iterating elementary constructions of geodesic parallelograms (Fig.~\ref{fig:schild_sphere}). The most famous, Schild's ladder (SL), was originally introduced by Alfred Schild in 1970 although no published reference exist\footnote{\cite{ehlers_geometry_1972} is often cited but no mention of the scheme is made in this work}. Its first appearance in the literature is in \cite{misner_gravitation_1973} where it is used to introduce the Riemannian curvature. A proof that the construction of a geodesic parallelogram ---i.e.\ one \textit{rung} of the ladder--- is an approximation of parallel transport was first given in \cite{kheyfets_schilds_2000}. However there is currently no formal proof that it converges when iterated along the geodesic along which the vector is to be transported, as prescribed in every description of the method, and because only the first order is given in \cite{kheyfets_schilds_2000}, Schild's ladder is considered to be a first-order method in the literature. 

\begin{figure}
    \centering
    \includegraphics[width=0.45\textwidth]{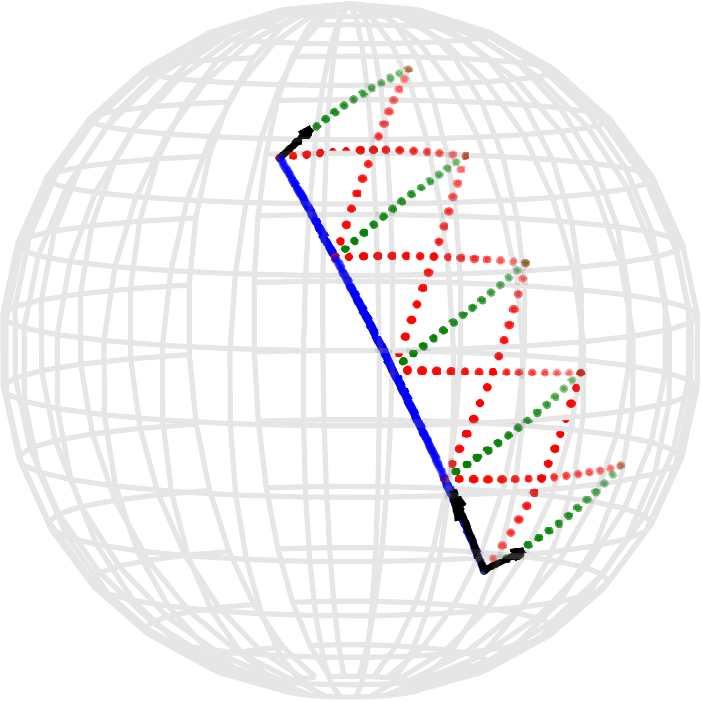}
    \caption{Schild's Ladder on the Sphere. The main geodesic is in blue, while the sides of the intermediate geodesic parallelograms are in dotted green and the diagonals in dotted red. The tangent vectors are represented by (scaled) black arrows. The construction is detailed in sec.~\ref{sec:elem_construction}}
    \label{fig:schild_sphere}
\end{figure}

The first aim of this paper is at filling this gap. Using a Taylor expansion of the \textit{double exponential} introduced by \cite{gavrilov_algebraic_2006, gavrilov_affine_2013}, we compute an expansion of the SL construction up to order three, with coefficients that depend on the curvature tensor and its covariant derivatives. 
Understanding this expansion then allows to give a proof of the convergence of the numerical scheme with iterated constructions, and to improve the scheme to a second-order method.
We indeed prove that an arbitrary speed of convergence in $\frac{1}{n^\alpha}$ can be obtained where $1\leq \alpha \leq 2$ and $n$ is the number of parallelograms, or rungs of the ladder.
This improves over the assumption made in \cite{louis_fanning_2018} that only a first-order convergence speed can be reached with SL. These results are observed with a high accuracy in numerical experiments performed using the open-source Python package \url{geomstats}\footnote{available at \url{http://geomstats.ai}} \cite{miolane_geomstats_2020}.

A slight modification of this scheme, Pole Ladder (PL), was introduced in \cite{lorenzi_efficient_2014}. It turns out that this scheme is exact in symmetric spaces, and in general that an elementary construction of PL is even more precise by an order of magnitude than that of SL \cite{pennec_parallel_2018}.
With the same method as for SL, we give a new proof of this result. By applying the previous analysis to PL, we demonstrate that a convergence speed of order two can be reached. Furthermore, we introduce a new construction which consists in averaging two SL constructions, and show that it is of the same order as the PL.

In most cases however, the exponential and logarithm maps are not available in closed form, and one also has to resort to numerical integration schemes. As for the Fanning Scheme (FS) \cite{louis_fanning_2018}, we prove that the ladder methods converge when using approximate geodesics and that all geodesics of the construction may be computed in one pass --- i.e.\ using one integration step (e.g. Runge-Kutta) per parallelogram construction, thus reducing the computational cost. We study the FS under the hood of ladder methods and show that  its implementation make it very close to SL. However, their elementary steps differ at the third order, so that the FS cannot be improved to a second-order method like SL or the PL.

To observe the impact of the different geometric structures on the convergence, we study the Lie group of isometries of $\R^d$, the special Euclidean group $SE(d)$. Endowed with a left-invariant metric, this space is a Riemannian symmetric space if the metric is isotropic \cite{zefran_generation_1998}. However, we show that it is no longer symmetric when using anisotropic metrics, and that geodesics may be computed by integrating the Euler-Poincar\'e equations. The same implementation thus allows to observe both geometric structures.
Furthermore, the curvature and its derivative may be computed explicitly \cite{milnor_curvatures_1976} and confirm our predictions. We treat this example in detail to demonstrate the impact of curvature on the convergence, and the code for the computations, and all the experiments of the paper are available online at \url{github.com/nguigs/ladder-methods}.

The first part of this paper is dedicated to Schild's ladder, while the second part applies the same methodology to the modified constructions. All the results are illustrated by numerical experiments along the way. The main proofs are in the text, and remaining details can be found in the appendices.

\subsection{Notations and Assumptions}

We consider a complete Riemannian manifold $(\M, g)$ of finite dimension $d \in \N$. The associated Levi-Civita connection defines a covariant derivative $\nabla$ and the parallel transport map. Denote $\exp$ the Riemannian exponential map, and $\log$ its inverse, defined locally. For $x \in \M$, let $T_x\M$ be the tangent space at $x$. The map $\exp_x$ sends $T_x\M$ to (a subset of) $\M$, and we will often write $x_w = \exp_x(w) \in \M$ for $w \in T_x\M$. 

Let $\gamma : [0,1] \rightarrow \M$ be a smooth curve with $\gamma(0) = x$. For $v \in T_x\M$, the parallel transport $\Pi_{\gamma, 0}^{t}v \in T_{\gamma(t)}\M$ of $v$ along $\gamma$ is defined as the unique solution at time $t \leq 1$ to the ODE $\nabla_{\dot \gamma(s)}X(s) = 0$ with $X(0) = v$.
In general, parallel transport depends on the curve followed between two points. In this work however, we focus on the case where $\gamma$ is a geodesic starting at $x$, and let $w$ be its initial velocity, i.e. $\gamma(t) = \exp_x(tw)$. Thus the dependence on $\gamma$ in the notation $\Pi_{\gamma}$ will be omitted, and we instead write $\Pi_x^y$ for the parallel transport along the geodesic joining $x$ to $y$ when it exists and is unique. The methods developed below can be extended straightforwardly to piecewise geodesic curves.

We denote by $\|\cdot\|$ the norm defined on each tangent space by the metric $g$ and $R$ the Riemann curvature tensor that maps for any $x \in \M$, $u, v, w \in T_x\M$ to $R(u,v)w \in T_x\M$. Throughout this paper, we consider $\gamma$ to be contained in a compact set $K \subset \M$ of diameter $\delta >0$, and thus $R$ and all its covariant derivatives $\nabla^n R$ can be uniformly bounded. 

\subsection{Double exponential and Neighboring logarithm}
\label{sec:double_exp}

We now introduce the main tool to compute a Taylor approximation of the ladder constructions, the double exponential (also written $\exp$) \cite{gavrilov_algebraic_2006, gavrilov_affine_2013}. It is defined for $x \in \M$, $(v, w) \in (T_x\M)^2$ by first applying $\exp_x$ to $v$, then $\exp_{x_v}$ to the parallel transport of $w$ along the geodesic from $x$ to $x_v$:
\begin{equation*}
    \exp_x(v, w) = \exp_{x_v}(\Pi_x^{x_v}w).
\end{equation*}

As the composition of smooth maps, it is also a smooth map. As the exponential map is locally one-to-one and the parallel transport is an isomorphism, we may define the function $h_x: U_x \rightarrow T_x\M$ on an open neighborhood $U_x$ of $(0,0) \in T_x\M \times T_x\M$ such that $\exp_x(U_x)$ is contained in a convex neighborhood of $\M$ and (Fig.~\ref{fig:double_exp}):
\begin{equation*}
    \exp_x(h_x(v, w)) = \exp_{x}(v, w) \quad \forall v,w \in T_x\M.
\end{equation*}

\begin{figure}
    \centering
    \includegraphics[width=0.7\textwidth]{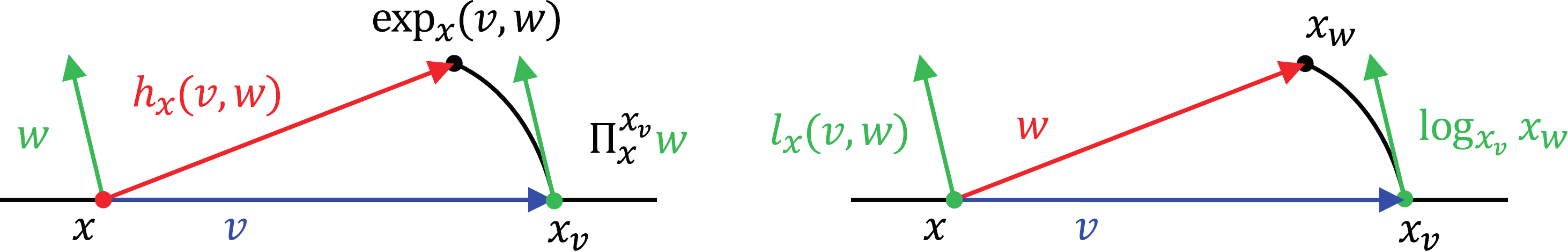}
    \caption{The double exponential and its inverse (left) and the neighboring logarithm (right) represented in a normal coordinate system centred at x. Geodesics are in black and tangent vectors are coloured}
    \label{fig:double_exp}
\end{figure}
As it is a smooth map, we can apply the Taylor theorem to $(s,t) \mapsto h_x(sv, tw)$. As explained in \cite{gavrilov_algebraic_2006, gavrilov_affine_2013}, its derivative are invariants of the connection that can be expressed in terms of the curvature and its covariant derivatives.

In this work we will require an approximation at order four, for $s,t\in \R$ small enough:
\begin{align}
    h_x(sv, tw) &= \log_x(\exp_{x_{sv}}(\Pi_x^{x_{sv}} tw)) \nonumber\\
              &= sv + tw +\frac{st}{6} R(w,v)(sv + 2tw) \nonumber\\
              &\quad + \frac{st}{24}\Big( (\nabla_v R)(w,sv)(5 tw + 2sv)\nonumber\\
              &\quad+ (\nabla_{w} R)(tw,v)(sv + 2tw) \Big) + q_5(sv,tw),
    \label{eq:taylor_double_exp}
\end{align}
where $q_5(sv,tw)$ contains homogeneous terms of degree five and higher, whose coefficients can be bounded uniformly in $K$ (as they can be expressed in terms of the curvature and its covariant derivatives). Because we consider in this paper the Levi-Civita connection of a Riemannian manifold, there is no torsion term appearing here.
Thus, taking $s,t\sim \frac{1}{n}$ for some $n \in \N$ large enough, $q_5(sv,tw) = O(\frac{1}{n^5})$. To simplify the notations, we write $O(5)$ in that sense. Formally \eqref{eq:taylor_double_exp} is similar to the BCH formula in Lie groups.

Similarly, \cite{pennec_curvature_2019} introduced the \textit{neighboring log} and computed its Taylor approximation. It is defined by applying the $\exp_x$ map to small enough $v,w \in T_x\M$ to obtain $x_x, x_w$ then computing the log of $x_w$ from ${x_v}$, and finally parallel transporting this vector back to $x$ (see Fig.~\ref{fig:double_exp}).
This defines $l_x: U_x \rightarrow T_x\M$ by:
\begin{equation*}
    l_x(v,w) = \Pi_{x_v}^x\log_{x_v}({x_w}),
\end{equation*}
which relates to the double exponential by solving
\begin{equation*}
    w = h_x(v, l_x(v,w)).
\end{equation*}

Using \eqref{eq:taylor_double_exp}, one can solve for the first terms of a Taylor expansion and obtains for $(v,w) \in U_x \subset T_x\M \times T_x\M$ \cite{pennec_curvature_2019}:
\begin{align}
    l_x(v, w) &= \Pi_{x_v}^x\log_{x_v}(x_w) \nonumber\\
              &= w - v +\frac{1}{6} R(v,w)(2w - v) \nonumber\\
              &\quad + \frac{1}{24}\Big( (\nabla_v R)(v,w)(3w - 2v) \nonumber\\
              &\quad+ (\nabla_{w} R)(v,w)(2 w - v) \Big) + O(5),
    \label{eq:taylor_double_log}
\end{align}
where we have implicitly assumed $v,w$ small enough and taken the time variables $s=t=1$. We will do so in the next section to simplify the notations.

\section{Schild's ladder}
We now turn to the construction of Schild's ladder and to the analysis of this numerical scheme.

\subsection{Elementary Construction}
\label{sec:elem_construction}
The construction to parallel transport $v \in T_x\M$ along the geodesic $\gamma$ with $\gamma(0) = x$ and $\dot \gamma(0) = w \in T_x\M$ (such that $(v,w) \in U_x$) is given by the following steps (Fig.~\ref{fig:schild_construction}):
\begin{enumerate}
    \item Compute the geodesics from $x$ with initial velocities $v$ and $w$ until time $s=t=1$ to obtain $x_v$ and $x_w$. These are the sides of the parallelogram.
    \item Compute the geodesic between $x_v$ and $x_w$ and the midpoint $m$ of this geodesic. i.e.\
    \begin{equation*}
        m = \exp_{x_v}\big(\frac{1}{2} \log_{x_v}(x_w)\big).
    \end{equation*}
    This is the first diagonal of the parallelogram.
    \item Compute the geodesic between $x$ and $m$, let $a \in T_x\M$ be its initial velocity. Extend it beyond $m$ for the same length as between $x$ and $m$ to obtain $z$, i.e.\
    \begin{equation*}
        a = \log_x(m); \qquad z = \exp_x(2a) = x_{2a}.
    \end{equation*}
    This is the second diagonal of the parallelogram.
    \item Compute the geodesic between $x_w$ and $z$. Its initial velocity $u^w$ is an approximation of the parallel transport of $v$ along the geodesic from $x$ to $x_w$, i.e.\
    \begin{equation*}
        u^w = \log_{x_w}(x_{2a}).
    \end{equation*}
\end{enumerate}
 By assuming that there exists a convex neighborhood that contains the entire parallelogram, all the above operations are well defined. In the literature, this construction is then iterated along $\gamma$ without further precision.

\begin{figure}
    \centering
    \includegraphics[width=0.45\textwidth]{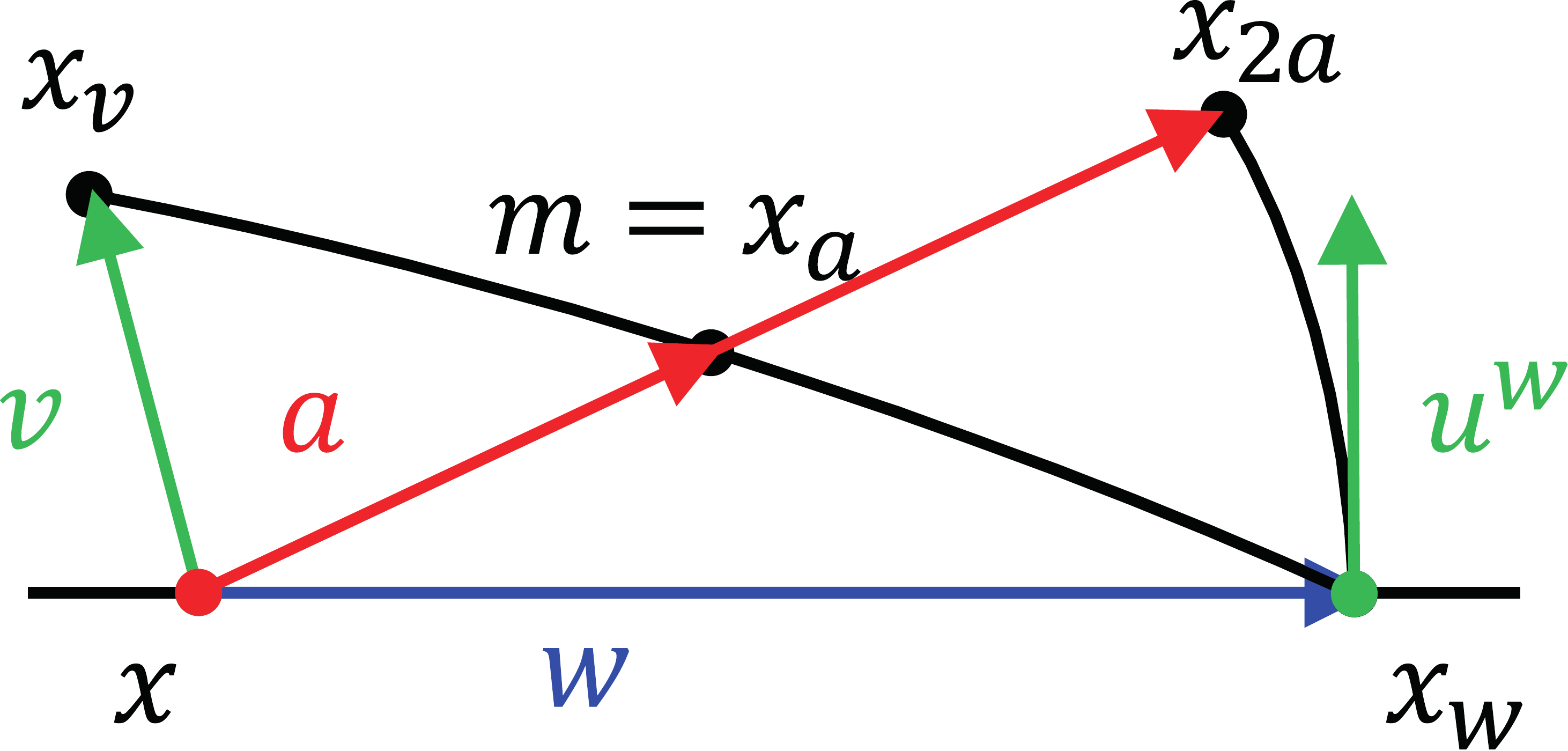}
    \caption{Construction of a geodesic parallelogram in Schild's ladder, represented in a normal coordinate system centred at $x$. Geodesics are in black and can be identified with tangent vectors (colored) when going through $x$.}
    \label{fig:schild_construction}
\end{figure}

\subsection{Taylor expansion}
\label{sec:taylor_schild}
We can now reformulate this elementary construction in terms of successive applications of the double exponential and the neighboring logarithm maps. Let $b^v=\log_{x_v}(x_w) \in T_{x_v}\M$ be the initial velocity of the geodesic computed in step 2 of the above construction, and $b$ its parallel transport to $x$. Therefore, we have $b = \Pi_{x_v}^x \log_{x_v}(x_w) = l_x(v,w)$. The midpoint is now computed by $m = \exp_{x_v}(\frac{b^v}{2})$, i.e.
\begin{equation}
    m = \exp_{x_v}(\Pi^{x_v}_x \frac{b}{2} )= h_x(v, \frac{b}{2}).
\end{equation}
Thus step 3. is equivalent to $a = h_x(v, \frac{1}{2} l_x(v,w))$. Combining the Taylor approximations \eqref{eq:taylor_double_exp} and \eqref{eq:taylor_double_log}, we obtain an expansion of $2a$. The computations are detailed in appendix~\ref{appendix_taylor_schild}, and we only report here the third order, meaning that all the terms of the form $\nabla_\cdot R(\cdot, \cdot)\cdot$ are summarized in the term $O(4)$:
\begin{equation}
    2a = w + v + \frac{1}{6} R(v, w)(w - v) + O(4) \label{eq:2a}
\end{equation}
We notice that this expression is symmetric in $v$ and $w$, as expected. Furthermore, the deviation from the Euclidean mean of $v,w$ (the parallelogram law) is explicitly due to the curvature. Accounting for this correction term is a key ingredient to reach a quadratic speed of convergence.

Now, in order to compute the error $e_x$ made by this construction to parallel transport $v$, $e_x=\Pi_x^{x_w}v - u^w$, we parallel transport it back to $x$: define $u = \Pi^x_{x_w}u^w$ and $e = \Pi^x_{x_w} e_x = v - u\in T_x\M$. Now as $u^w = \log_{x_w}(x_{2a})$, we have $u = l_x(w, 2a)$. Combining \eqref{eq:taylor_double_log} with \eqref{eq:2a} (see appendix~\ref{appendix_taylor_schild}), we obtain the first main result of this paper, a third order approximation of the Schild's ladder construction: 

\begin{theorem}
\label{thm:shild}
Let $(\M, g)$ be a finite dimensional Riemannian manifold. Let $x\in \M$ and $v,w \in T_x\M$ sufficiently small. Then the output $u$ of one step of Schild's ladder parallel transported back to $x$ is given by
\begin{equation}
    u = v + \frac{1}{2} R(w, v)v + O(4) \label{eq:taylor_schild}
\end{equation}
\end{theorem}
The fourth order and a bound on the remainder are detailed in appendix~\ref{appendix_taylor_schild}. This theorem shows that Schild's ladder is in fact a second-order approximation of parallel transport. Furthermore, this shows that splitting the main geodesic into segments of length $\frac{1}{n}$ and simply iterating this construction $n$ times will in fact sum $n$ error terms of the form $R(\frac{w_i}{n},v_i)v_i$, hence by linearity of $R$, the error won't necessarily vanish as $n \rightarrow \infty$. To ensure convergence, it is necessary to also scale $v$ in each parallelogram. The second main contribution of this paper is to clarify this procedure and to give a proof of convergence of the numerical scheme when scaling both $v$ and $w$. This is detailed in the next subsection.

\subsection{Numerical Scheme and proof of convergence}
With the notations of the previous subsection, let us define $\schild(x, w, v) = u^w \in T_{x_w}M$. We now divide the geodesic $\gamma$ into segments of length $\frac{1}{n}$ for some $n\in \N^*$ large enough so that the previous Taylor expansions can be applied to $\frac{w}{n}$ and $\frac{v}{n}$. As mentioned before, $v$ needs to be scaled as $w$. In fact let $\alpha \geq 1$ and consider the sequence defined by (see Fig.~\ref{fig:shild_iterated_construction})
\begin{align}
    \label{eq:sequence_schild}
    v_0 &= v \nonumber \\
    v_{i+1} &= n^\alpha \cdot \schild(x_i, \frac{w_i}{n}, \frac{v_i}{n^\alpha}),
\end{align}
where $x_i = \gamma(\frac{i}{n})=\exp_x \big( \frac{i}{n} w \big)$, $w_i = n \log_{x_i}(x_{i + 1}) = \Pi_x^{x_i} w$. We now establish the following result, which ensures convergence of Schild's ladder to the parallel transport of $v$ along $\gamma$ at order at most two.

\begin{theorem}
\label{thm:convergence_schild}
Let $(x_i, v_i, w_i)_{(i\leq n)}$ be the sequence defined as above. Then 
$\exists \tau > 0, \exists \beta >0, \exists N \in \N, \forall n > N$,
\[\|v_n - \Pi^{x_n}_x v\| \leq \frac{\tau}{n^\alpha} + \frac{\beta}{n^2}.\]
Moreover, $\tau$ is bounded by a bound on the sectional curvature, and $\beta$ by a bound on the covariant derivative of the curvature tensor.
\end{theorem}

\begin{figure}
    \centering
    \includegraphics[width=0.7\textwidth]{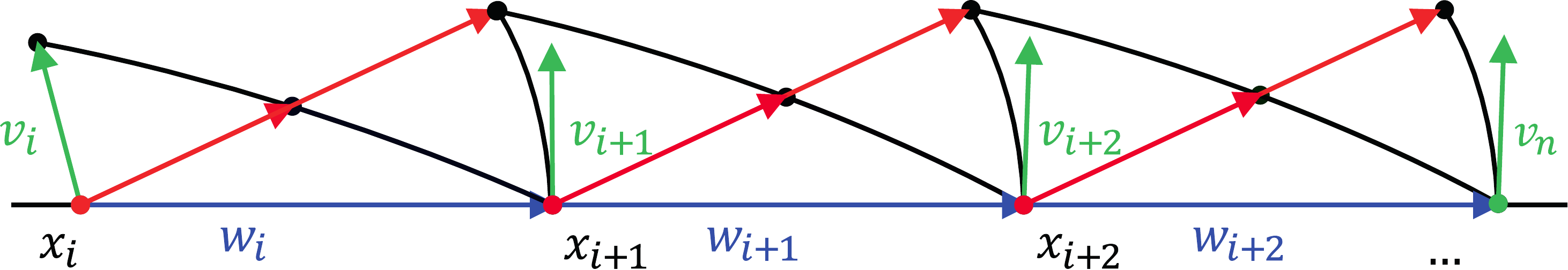}
    \caption{Schild's ladder, consisting in iterative constructions of geodesic parallelograms}
    \label{fig:shild_iterated_construction}
\end{figure}

\begin{proof}
To compute the accumulated error, we parallel transport back to $x$ the results $v_n$ of the algorithm after $n$ rungs. The error is then written as the sum of the errors at each rungs, parallel transported to $x$ (isometrically):
\begin{align}
      \frac{1}{n^\alpha} \| v_0 - \Pi_{x_n}^x v_n \| &= \| \sum_{i=0}^{n-1} \Pi_{x_i}^x \big( \frac{v_i}{n^\alpha} - \Pi_{x_i+1}^{x_i} \frac{v_{i+1}}{n^\alpha} \big) \| \nonumber\\
      &\leq \sum_{i=0}^{n-1} \|\Pi_{x_i}^x \big( \frac{v_i}{n^\alpha} - \Pi_{x_i+1}^{x_i} \frac{v_{i+1}}{n^\alpha} \big) \| \leq \sum_{i=0}^{n-1} \|e_i \|. \label{eq:error}
\end{align}
By \eqref{eq:taylor_schild} and lemma~\ref{lemma:r4}, the error at each step can be written
\begin{equation}
    e_i = \frac{v_i}{n^\alpha} - \Pi_{x_i+1}^{x_i} \frac{v_{i+1}}{n^\alpha} = \frac{1}{2n^{(1 + 2\alpha)}} R(v_i, w_i)v_i + r_4(\frac{w_i}{n},\frac{v_i}{n^\alpha}). \label{eq:ei}
\end{equation}
where $r_4$ contains the fourth order residual terms, and is given in appendix~\ref{appendix_taylor_schild}. We start as in \cite{louis_computational_2019} by assuming that $v_i$ doesn't grow too much, i.e.\ there exists $s \in \{1,\ldots,n\}$ such that $\forall i < s, \|v_i\| \leq 2 \|v\|$. We will then show that the control obtained on $v_n$ is tight enough so that when $n$ is large enough, $\forall k \leq n, \|v_k\| \leq 2\|v\|$. With this assumption each term of the right-hand side (r.h.s.) can be bounded. As $\alpha \geq 1$, we apply lemma~\ref{lemma:r4} (in appendix~\ref{appendix_taylor_schild}): $\exists \beta > 0$ such that for $n$ large enough:
\[\forall i \leq n, \|r_4(v_i, w_i)\| \leq \frac{\beta}{n^{\alpha + 3}}.\]
Similarly, as $\forall i \leq n, \; \|w_i\|=\|w\|$ and by assumption $\|v_i\| \leq 2 \|v\|$, we have
\begin{equation*}
    \|R(v_i, w_i)v_i\| \leq \| R\|_\infty \|v_i\|^{2} \| w_i \| \leq 4 \|v\|^2 \|R\|_\infty \| w \|
\end{equation*}
Let $\tau = 2 \|v\|^2 \|R\|_\infty \|w \|$. This gives
\begin{equation}
    \|e_i\| \leq \frac{\tau}{n^{(1 + 2\alpha)}} + \frac{\beta}{n^{\alpha + 3}}.
\end{equation}
As the r.h.s.\ does not depend on $n$, it can be plugged into \eqref{eq:error} to obtain by summing for $i=0,\ldots,s-1 \leq n$:
\begin{equation*}
      \frac{1}{n^\alpha}\| v_0 - \Pi_{x_s}^x v_s \| \leq \sum_{i=0}^{s-1} \|e_i \|
      \leq \frac{\tau}{n^{2\alpha}} + \frac{\beta}{n^{\alpha + 2}},
\end{equation*}
and finally
\begin{equation}
    \label{eq:control}
    \|v - \Pi_{x_s}^x v_s\| \leq \frac{\tau}{n^\alpha} + \frac{\beta}{n^2}.
\end{equation}
Now suppose for contradiction that for arbitrarily large $n$, there exists $k \leq n$ such that $\|v_k\| > 2\|v\|$ and choose $u(n) \in {1,\ldots,n}$ minimal with this property, i.e.\ $\|v_{u(n)}\| > 2\|v\|$ so that \eqref{eq:control} can be used with $s=u(n)$. Then we have:
\begin{equation}
    \|v\| < \big| \|v\| - \|v_{u(n)}\|\big| \leq \|v - \Pi_{x_{u(n)}}^x v_{u(n)}\| \leq \frac{\tau}{n^\alpha} + \frac{\beta}{n^2}.
\end{equation}
But the r.h.s.\ goes to $0$ as $n$ goes to infinity, leading to a contradiction. Therefore, for $n$ large enough, $\forall i \leq n$, $\|v_i\| \leq 2 \|v\|$, and the previous control on $\|v - \Pi_{x_s}^x v_s\|$ given by \eqref{eq:control} is valid for $s=n$. The result follows by parallel transporting this error to $x_n$, which doesn't change its norm.
\qed
\end{proof}

\begin{remark}
\label{remark:dominant}
The proof allows to grasp the origin of the bounds in Thm.~\ref{thm:shild}. The term $O(n^{-\alpha})$ comes from $R(w,v)v$ in \eqref{eq:taylor_schild}, as it is bilinear in $v$, so that the division by $n^\alpha$ is squared in the error, while we only need to multiply by $n^\alpha$ at the final step to recover the parallel transport of $v$. Similarly, as $R(w,v)v$ is linear in $w$, the division by $n$ and summing of $n$ terms doesn't appear in the error.

On the other hand, the $O(n^{-2})$ comes from $\nabla_w R(w,v)w$ (from lemma \ref{lemma:r4}). Indeed, this term is invariant by the division/multiplication by $n^\alpha$, unlike other error terms where $v$ appears several times, and the multilinearity w.r.t.\ $w$ implies that it is a $O(n^{-3})$ that is then summed $n$ times. We therefore use $\alpha=2$ in practice.
\end{remark}

\begin{remark}
The bound on the curvature that we use can be related to a bound on the sectional curvature $\kappa$ as follows. Recall
\begin{equation}
    \kappa(X,Y) = \frac{<R(Y,X)X,Y>}{\|X\|^2 \|Y\|^2 - <X,Y>^2}.
\end{equation}
Suppose it is bounded on the compact set $K$: $\forall X,Y \in \Gamma(K), |\kappa(X,Y)| < \|\kappa\|_\infty$
Therefore, for two linearly independent $X,Y$, write $\delta_{X,Y} = \frac{<X,Y>^2}{\|X\|^2 \|Y\|^2} \in [0,1)$,
\begin{align*}
    |<R(Y, X)X, Y>| &= |\kappa(X, Y)| (\|X\|^2 \|Y\|^2 - <X,Y>^2)\\
                  &= |\kappa(X, Y)| \|X\|^2 \|Y\|^2 (1 - \delta_{X,Y})\\
                  &\leq \|\kappa\|_\infty \|X\|^2 \|Y\|^2.
\end{align*}
Now, the linear map $Y \mapsto R(Y, X)X$ is self-adjoint, so if $|\lambda_1(X)|\leq \ldots \leq |\lambda_d(X)|$ are its eigenvalues, we have on the first hand $|\lambda_d(X)| \leq \|\kappa\|_\infty \|X\|^2$ and on the second hand $\|R(Y,X)X\| \leq |\lambda_d(X)|\|Y\|$. Thus
\begin{equation}
    \|R(Y,X)X\| \leq \|\kappa\|_\infty \|X\|^2 \|Y\|.
\end{equation}
We then could have used $\tau = \frac{\delta^3 \|\kappa \|_\infty}{2}$ in the proof of theorem~\ref{thm:convergence_schild}.
\end{remark}

We notice in this proof and remark \ref{remark:dominant} that the terms of the Taylor expansion where $v$ appears several times vanish faster thanks to the arbitrary exponent $\alpha \geq 1$. Together with lemma~\ref{lemma:r4}, this implies that the dominant term is the one where $v$ appears only once (given explicitly in appendix~\ref{appendix_taylor_schild}), which imposes here a speed of convergence of order two. We will use this key fact to compute the speed of convergence of the other schemes.

\subsection{Numerical Simulations}
\label{sec:numerical_sl}
In this section, we present numerical simulations that show the convergence bounds in two simple cases: the sphere $S^2$ and the space of $3\times 3$ symmetric positive-definite matrices $SPD(3)$.

\paragraph{The Sphere}
We consider the sphere $S^2$ as the subset of unit vectors of $\R^3$, endowed with the canonical ambient metric $<\cdot,\cdot>$. It is a Riemannian manifold of constant curvature, and the geodesics and parallel transport are available in closed form (these are given in appendix~\ref{appendix:sphere} for completeness).

Because the sectional curvature $\kappa$ is constant, the fourth-order terms vanish and theorem~\ref{thm:convergence_schild} becomes $\|v - \Pi_{x_n}^x v_n\| \leq \frac{\tau}{n^\alpha}$ for $\alpha \geq 1$. This is precisely observed on figure~\ref{fig:schild_sphere_alpha}.

\begin{figure}
    \begin{minipage}{.49\textwidth}
        \centering
        \includegraphics[width=0.9\textwidth]{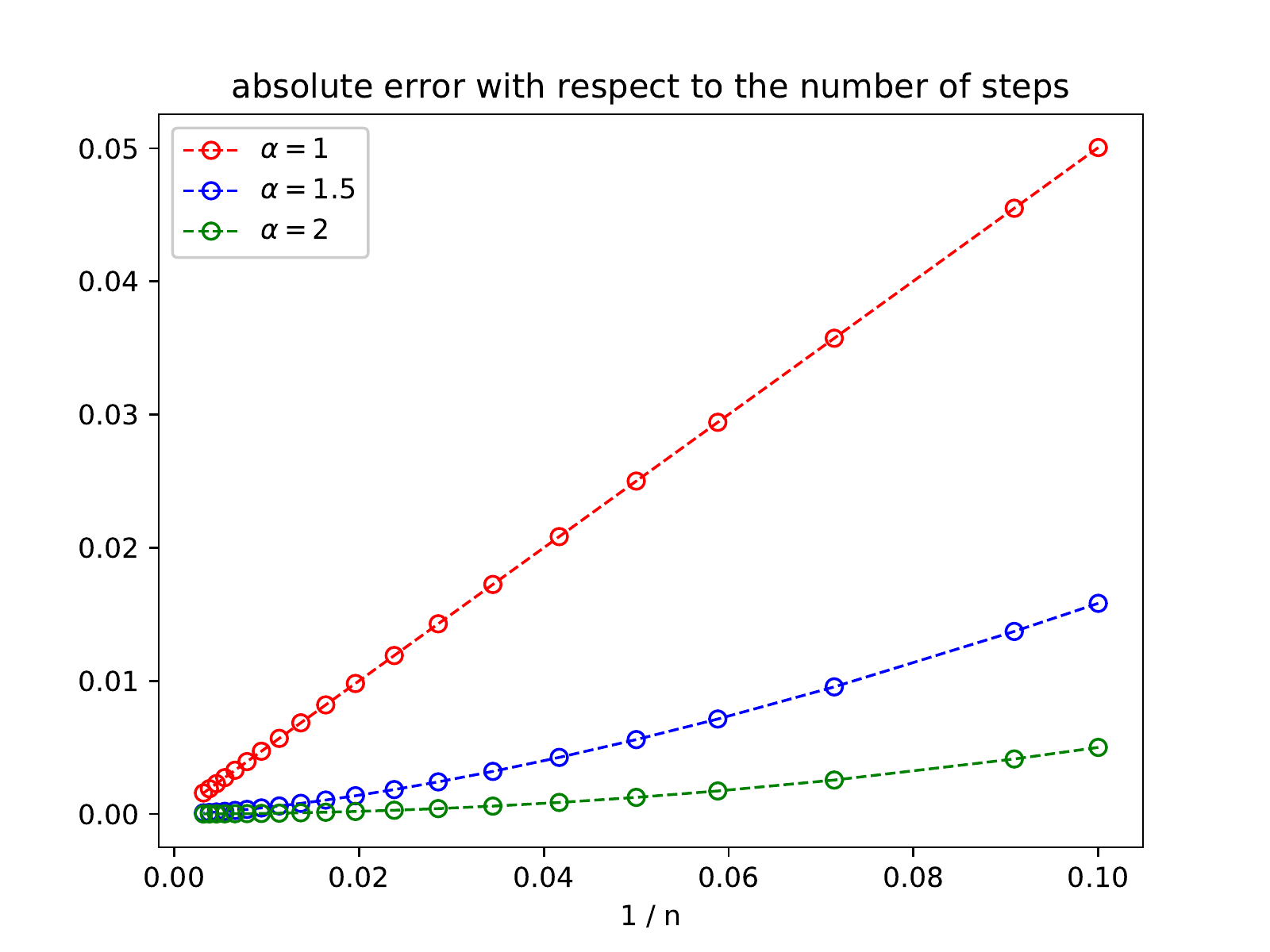}
    \end{minipage}
    \begin{minipage}{.5\textwidth}
        \centering
        \includegraphics[width=0.9\textwidth]{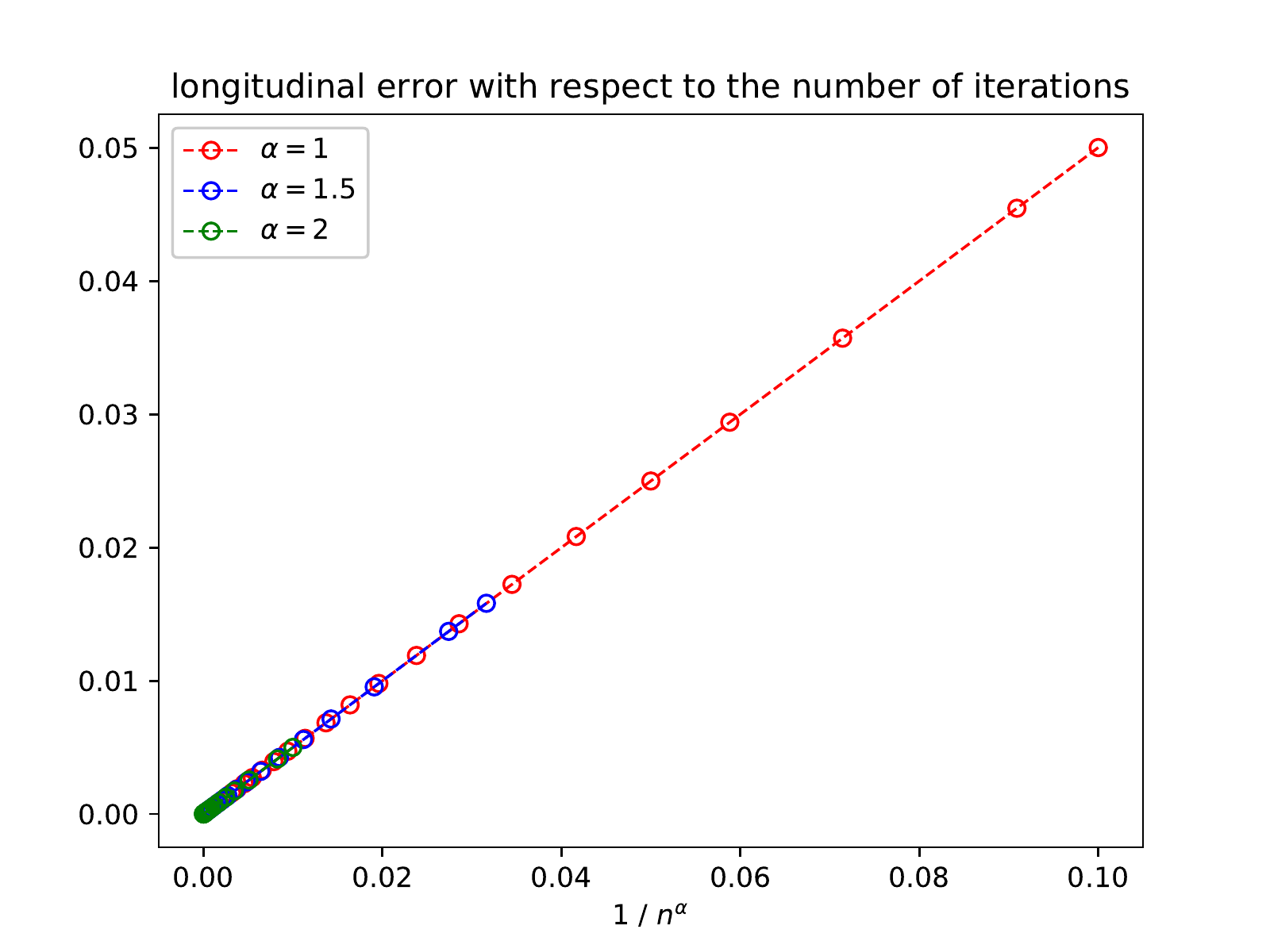}
    \end{minipage}
    \caption{Error of the parallel transport of $v$ along the geodesic with initial velocity $w$ where $v$ and $w$ are orthonormal. Accordingly with our main result, any speed of convergence $1 \leq \alpha \leq 2$ can be reached. Left : Absolute error w.r.t\ $\frac{1}{n}$. Right: Longitudinal error w.r.t.\ $\frac{1}{n^\alpha}$: on $S^2$ this is a straight line with slope $\frac{1}{2}$.}
    \label{fig:schild_sphere_alpha}
\end{figure}

Moreover, using Eq.(\ref{eq:ei}), we have at each rung:
\begin{align*}
    <\Pi_{x_i}^x e_i, w> &= <e_i, w_i> = \frac{1}{2n^{(1+2\alpha)}}<R(v_i, w_i)v_i, w_i> + O(\frac{1}{n^{(4 + \alpha)}})\\
            &= \frac{-1}{2n^{(1+2\alpha)}}\kappa(v_i, w_i) = \frac{-1}{2n^{(1+2\alpha)}} + O(\frac{1}{n^{(4 + \alpha)}}).
\end{align*}
Therefore, by summing as in the proof of theorem~\ref{thm:convergence_schild} we obtain: 
\begin{equation*}
    <v_n - \Pi^{x_n}_x v, w_n> = \frac{1}{2n^\alpha} + O(\frac{1}{n^3}).
\end{equation*}
Thus the projection of the error onto $w$, which we call longitudinal error, allows to verify our theory: we can measure the slope of the decay with respect to $n^{-\alpha}$ and it should equal $\frac{1}{2}$. This is very precisely what we observe (Fig.\ref{fig:schild_sphere_alpha}, right).

\paragraph{SPD} We now consider $SPD(3)$ endowed with the affine-invariant metric \cite{pennec_manifold-valued_2020}. Again, the formulas for the exp, log and parallel transport maps are available in closed form (appendix~\ref{appendix:spd}). Because the sectional curvature is always non-positive, the slope of the longitudinal error is negative, and this is observed on Fig.~\ref{fig:schild_spd}.

\begin{figure}
    \begin{minipage}{.49\textwidth}
        \centering
        \includegraphics[width=0.9\textwidth]{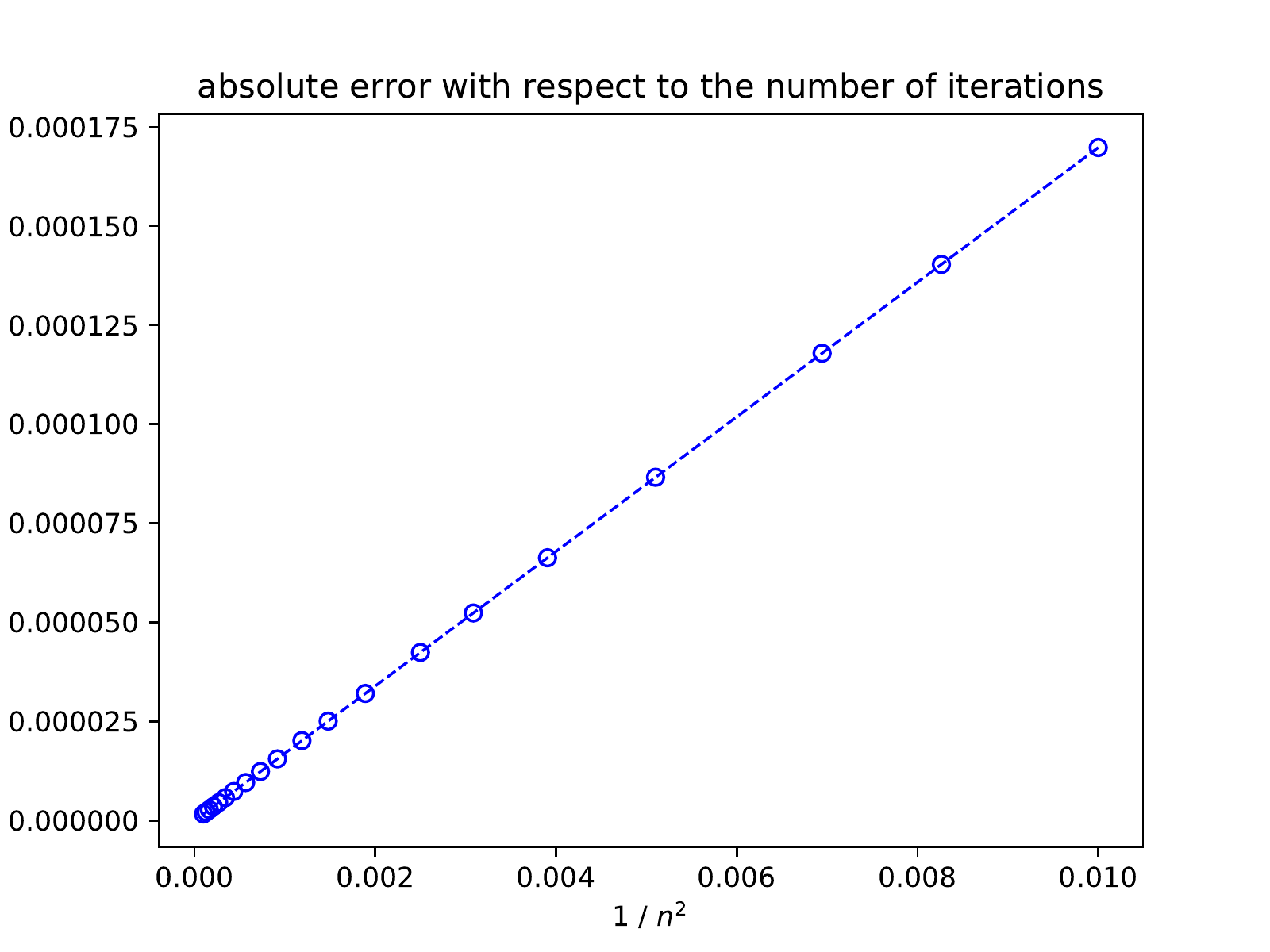}
    \end{minipage}
    \begin{minipage}{.5\textwidth}
        \centering
        \includegraphics[width=0.9\textwidth]{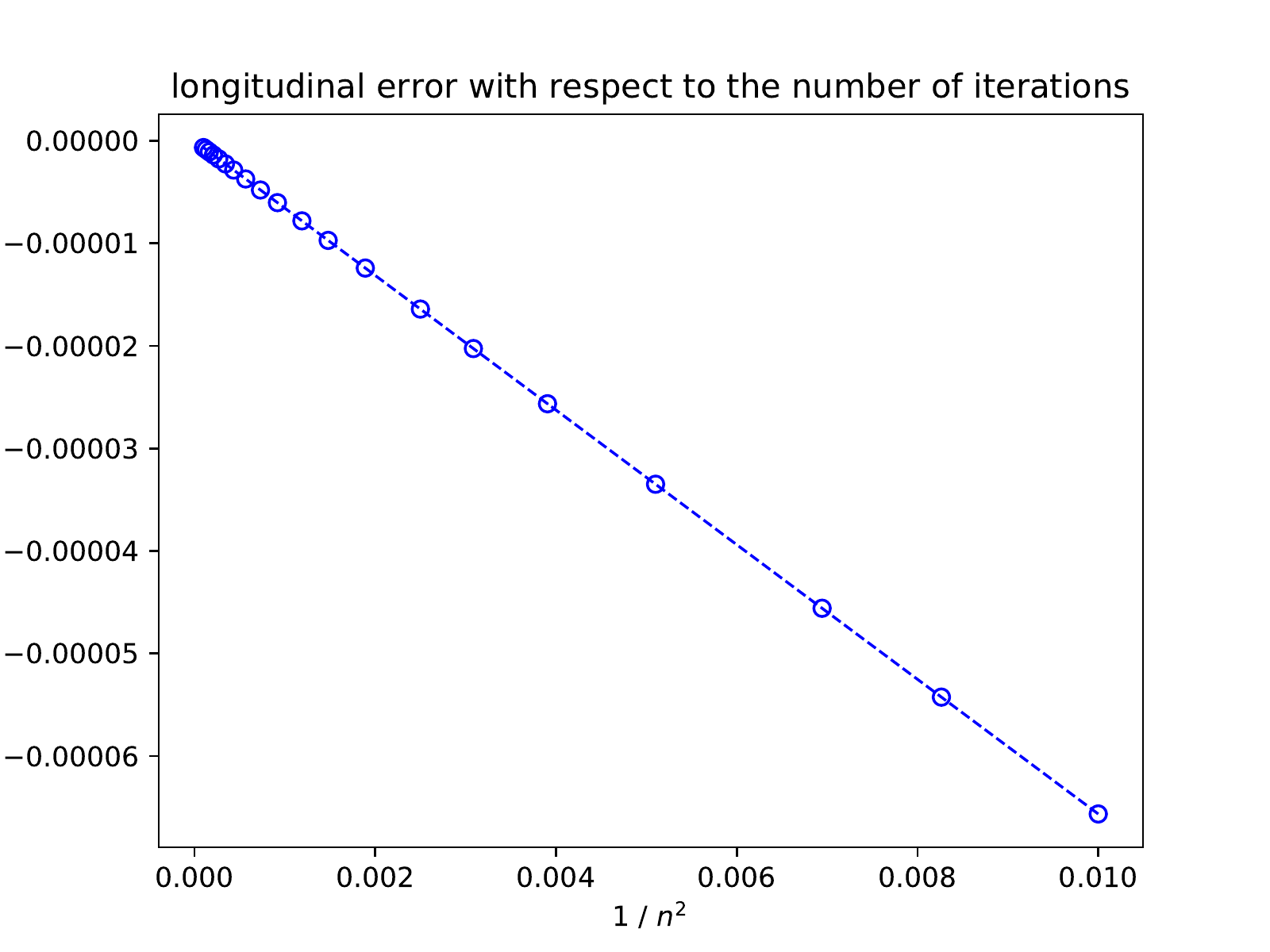}
    \end{minipage}
    \caption{Error of the parallel transport of $v$ along the geodesic with initial velocity $w$ where $v$ and $w$ are orthonormal. Accordingly with our main result, a quadratic speed of convergence $\alpha = 2$ can be reached. Left : Absolute error w.r.t\ $\frac{1}{n^2}$. Right: Longitudinal error w.r.t.\ $\frac{1}{n^2}$: with the AI metric, this is a straight line with negative slope.}
    \label{fig:schild_spd}
\end{figure}

\section{Variations of SL}

We now turn to variations of SL. First, we revisit the Fanning Scheme using the neighboring log and show how close to SL it in fact is. Then we examine the pole ladder, and introduce the averaged SL. Furthermore, we introduce infinitesimal schemes, where the exp and log maps are replaced by one step of a numerical integration scheme. We give convergence results and simulations for the PL in this setting. This allows to probe the convergence bounds in settings where the underlying space is not symmetric, and therefore observe the impact of a non-zero $\nabla R$.

\subsection{The Fanning Scheme, revisited}
The Fanning Scheme \cite{louis_fanning_2018} leverages an identity given in \cite{younes_jacobi_2007} between Jacobi fields and parallel transport. As proved in \cite{louis_fanning_2018}, it converges linearly when dividing the main geodesic into $n$ segments and using second-order integration schemes of the Hamiltonian equations. We recall the algorithm here with our notations and compare it to SL.
For $v,w \in T_x\M$, define the Jacobi field (JF) along $\gamma: t \mapsto \exp_x(tw)$ at time $h$ small enough by:
\begin{equation*}
    J_{\gamma(t)}^v(h) = \frac{\partial}{\partial \epsilon}\bigg|_{\epsilon=0}\exp_{\gamma(t)}\big(h(\dot \gamma(t) + \epsilon v)\big).
\end{equation*}
One can then show that $J_{\gamma(t)}^v(h) = h\Pi_{\gamma,t}^{t+h} v + O(h^2)$\cite{younes_jacobi_2007}.
The Fanning Scheme consists in computing \textit{perturbed} geodesics, i.e.\ with initial velocity $w + \epsilon v$ for some $\epsilon > 0$ and to approximate the associated JF by
\begin{equation}
    \label{eq:jac_finite}
    J_{\gamma(t)}^v(h) = \frac{\exp_{\gamma(t)}\big(h(\dot \gamma(t) + \epsilon v)\big) - \exp_{\gamma(t)}\big(h\dot \gamma(t)\big)}{\epsilon},
\end{equation}
and the parallel transport of $v$ along $\gamma$ between $0$ and $h$ by
\begin{equation*}
    u^w = \frac{\exp_x(h(w + \epsilon v)) - \exp_x(hw)}{h\epsilon}.
\end{equation*}
This defines an elementary construction that is iterated $n$ times with time-steps $h=\frac{1}{n}$, and \cite{louis_fanning_2018} showed that using $\epsilon=h$ yields the desired convergence. 
Let's consider that the finite difference approximation of \eqref{eq:jac_finite} is in fact a first-order approximation of the log of $\exp_x\big(h(w + \epsilon v)\big)$ from $\exp_x(hw)$. Now this scheme is similar to SL, except that it uses $2a = w + \epsilon v$ instead of \eqref{eq:2a}, i.e.\ the correction term according to curvature is not accounted for (compare figures \ref{fig:schild_construction} and \ref{fig:fanning}). When using $h = \epsilon= \frac{1}{n}$, this corresponds to $\alpha=2$ in our analysis of the sequence defined by~\ref{eq:sequence_schild}. Similarly, we can compute a third order Taylor approximation of the error made at each step of the FS. To do so, introduce as in sec.~\ref{sec:taylor_schild} $u = \Pi^x_{x_{hw}}u^w$. Then
\begin{align*}
    u &= l_x\big(hw, h(w + \epsilon v)\big) \\
      &= h\epsilon v + \frac{h^2}{6} R(w, w + \epsilon v)\left(2h (w + \epsilon v) - hw\right) + O(h^4\epsilon)\\
      &= h\epsilon v + \frac{h^3 \epsilon}{6} R(w,v)(w + 2 \epsilon v) + O(h^4\epsilon).
\end{align*}

\begin{figure}
    \centering
    \includegraphics[width=0.45\textwidth]{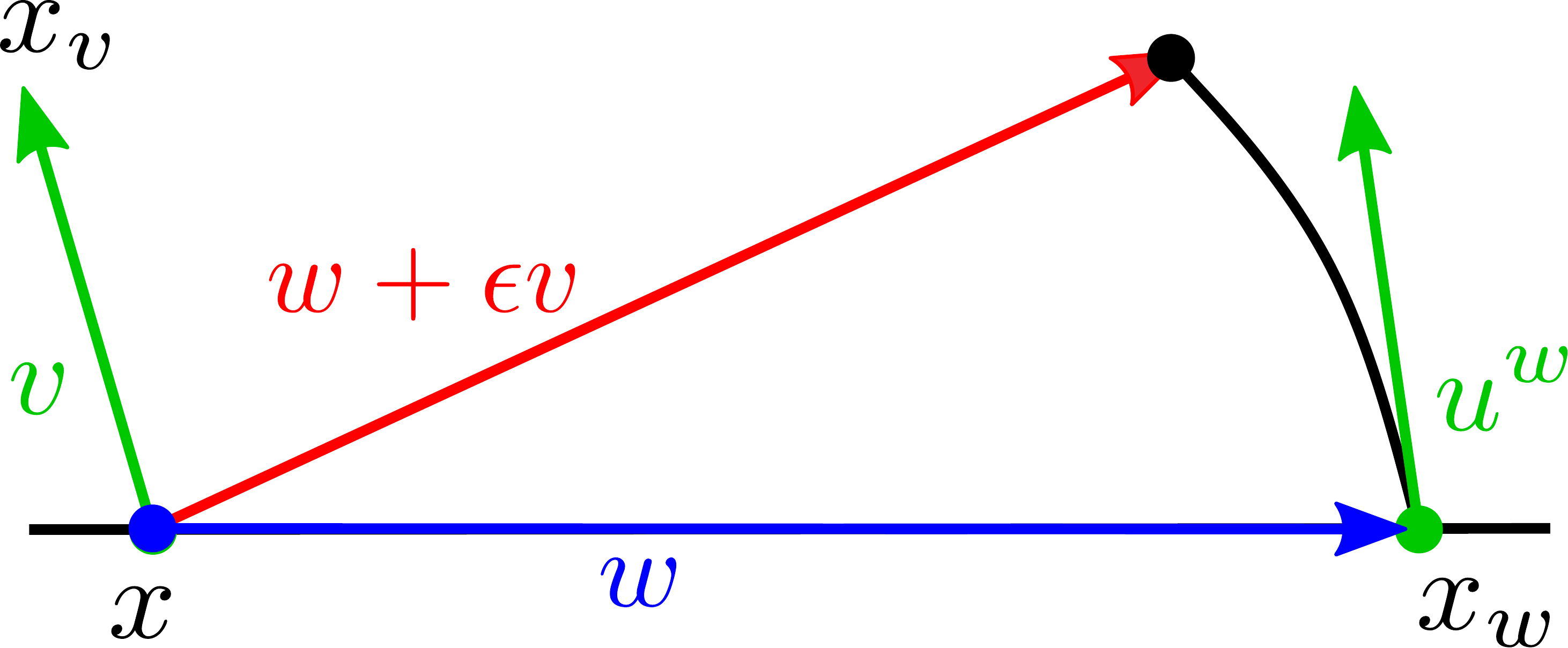}
    \caption{Elementary construction of the Fanning Scheme. For $\epsilon=\frac{1}{n}$, it is similar to that of SL (with $\alpha=2$) except that the midpoint is not computed but $w + \epsilon v$ is used instead of $\log_x(m)$.}
    \label{fig:fanning}
\end{figure}

But as $\epsilon = h$, the last term including $v$ disappears in the $O(h^4\epsilon)$ and we obtain:
\begin{equation}
    \label{eq:taylor_fs}
    \frac{u}{h\epsilon} = v + \frac{h^2}{6} R(w,v)w + O(h^3).
\end{equation}
We notice that unlike the expression given by theorem~\ref{thm:shild}, this approximation of FS contains a term linear in $v$ (i.e.\ where $v$ appears only once), and bilinear in $w$, so that when applied to $h=\frac{1}{n}$, summing error terms yields a global error of order $\frac{1}{n}$. This shows that a better speed cannot be reached. Moreover, we recognize the coefficient $\frac{1}{6n}$ computed explicitly on the sphere in \cite[sec. 2.3]{louis_fanning_2018}.
As in the previous section, define $u^w = \fs(x, w, v)$ the result of the FS construction, and the sequence
\begin{align}
    \label{eq:sequence_fs}
    v_0 &= v, \nonumber \\
    v_{i+1} &= n^2 \cdot \fs\left(x_i, \frac{w_i}{n}, \frac{v_i}{n^2}\right),
\end{align}
where $x_i = \gamma(\frac{i}{n})=\exp_x \big( \frac{i}{n} w \big)$, $w_i = n \log_{x_i}(x_{i + 1}) = \Pi_x^{x_i} w$. Then it is straightforward to reproduce the proof of thm.~\ref{thm:convergence_schild} with \eqref{eq:taylor_fs} instead of \eqref{eq:taylor_schild} to show that $\exists \beta >0, \exists N \in \N, \forall n > N$,
\[\|v_n - \Pi^{x_n}_x v\| \leq \frac{\beta}{n}.\]
This corroborates the result of \cite{louis_fanning_2018}, although at this points geodesics are assumed to be available exactly. An improvement of the FS is to use a second-order approximation of the Jacobi field by computing two perturbed geodesics, but this doesn't change the precision of the approximation of the parallel transport. We implemented this scheme with this improvement and the result is compared with ladder methods in paragraph~\ref{sec:complexity}.

\subsection{Averaged SL}
\label{sec:averaged}
Focusing on Eq.(\ref{eq:taylor_schild}) of Theorem~\ref{thm:shild}, we notice that the error term is even with respect to $v$. We therefore propose to average the results of a step of SL applied to $v$ and to $-v$ to obtain $u^{(+)}$ and $u^{(-)}$. Using \eqref{eq:schild_bis} of appendix~\ref{appendix_taylor_schild}, we have:
\begin{align}
    u^{(+)} &= v + \frac{1}{2} R(w, v)v + \frac{7}{48}\Big((\nabla_v R)(w,v)v + (\nabla_w R)(w,v)(\frac{8}{7}v + w)\Big) + O(5),\\
    u^{(-)} &= -v + \frac{1}{2} R(w, v)v - \frac{7}{48}\Big((\nabla_v R)(w,v)v + (\nabla_w R)(w,v)(w -\frac{8}{7}v)\Big) + O(5).
\end{align}
Hence
\begin{align}
    \frac{u^{(+)} - u^{(-)}}{2} &= v + \frac{7}{48}\Big( (\nabla_v R)(w,v)v + (\nabla_w R)(w,v)w\Big) + O(5).
\end{align}
This scheme thus allows to cancel out the third order term in an elementary construction. However, when iterating this scheme, the dominant term will be $(\nabla_w R)(w,v)w$, so that the speed of convergence remains quadratic anyway. Being also heavier to compute, this scheme has thus very little practical advantage to offer, while the pole ladder, detailed in the next section, present both advantages of being exact at the third order, and being much cheaper to compute.

\subsection{Pole Ladder}
The pole ladder was introduced by \cite{lorenzi_efficient_2014} to reduce the computational burden of SL. They showed that at the first order, the constructions where equivalent. \cite{pennec_parallel_2018} latter gave a Taylor approximation of the elementary construction showing that each rung of pole ladder is a third order approximation of parallel transport, and additionally showed that PL is exact in affine (hence in Riemannian) locally symmetric spaces (the proof is reproduced in \cite{guigui_symmetric_2019}). We first present the scheme and derive the Taylor approximation using the neighboring log.

\subsubsection{The elementary construction}

\begin{figure}
    \centering
    \includegraphics[width=0.7\textwidth]{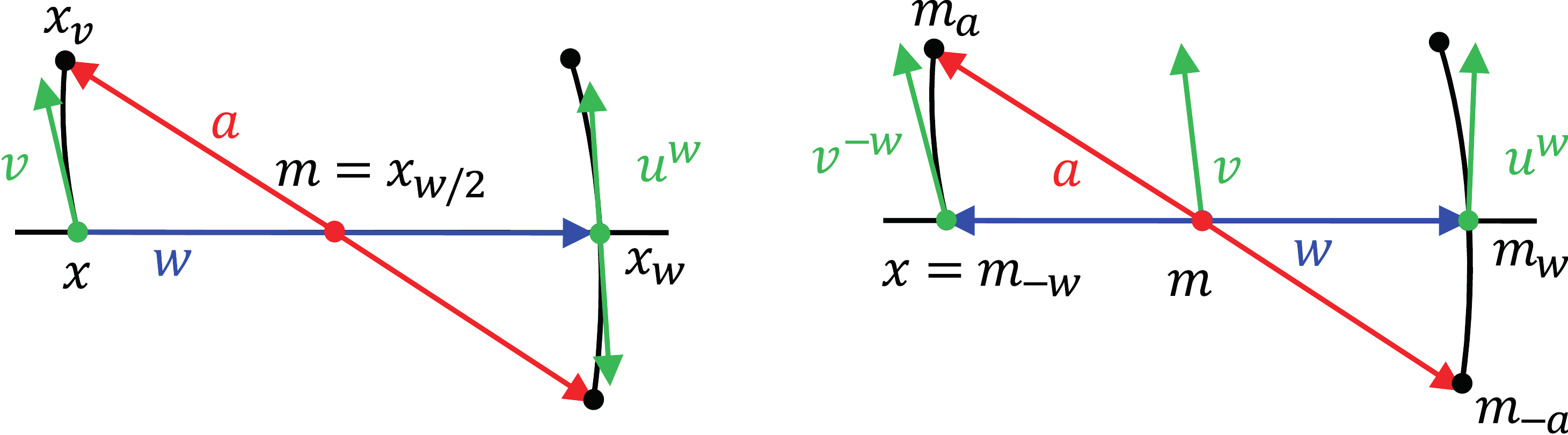}
    \caption{Elementary construction of the pole ladder with the previous notations (left), and new notations (right), in a normal coordinate system at $m$.}
    \label{fig:pole_construction}
\end{figure}

The construction to parallel transport $v \in T_x\M$ along the geodesic $\gamma$ with $\gamma(0) = x$ and $\dot \gamma(0) = w \in T_x\M$ (such that $(v,w) \in U_x$) is given by the following steps (see Fig.~\ref{fig:pole_construction}):
\begin{enumerate}
    \item Compute the geodesics from $x$ with initial velocities $v$ and $w$ until time $s=t=1$ to obtain $x_v$ and $x_w$, and $t=\frac{1}{2}$ to obtain the midpoint $m$. The main geodesic $\gamma$ is a diagonal of the parallelogram.
    \item Compute the geodesic between $x_v$ and $m$, let $a \in T_m\M$ be its initial velocity. Extend it beyond $m$ for the same length as between $x_v$ and $m$ to obtain $z$, i.e.\
    \begin{equation*}
        a = \log_m(x_v); \qquad z = \exp_m(-a) = m_{-a}.
    \end{equation*}
    This is the second diagonal of the parallelogram.
    \item Compute the geodesic between $x_w$ and $z$. The opposite initial velocity $u^w$ is an approximation of the parallel transport of $v$ along the geodesic from $x$ to $x_w$, i.e.\
    \begin{equation*}
        u^w = -\log_{x_w}(m_{-a}).
    \end{equation*}
\end{enumerate}
 By assuming that there exists a convex neighborhood that contains the entire parallelogram, all the above operations are well defined.

\subsubsection{Taylor approximation}
We now introduce new notations, and centre the construction at the midpoint $m$ (see Fig.~\ref{fig:pole_construction}, right). Both $v,w$ are now tangent vectors at $m$, and $v$ is the parallel transport of the vectors $v^{-w}$ that we wish to transport between $m_{-w}$ and $m_w$. With these new notations,
\begin{equation*}
    u = \Pi_{m_w}^m u^w = - l_m(w, -a)
\end{equation*}
where $a=h_m(-w, v)$. Therefore, using \eqref{eq:taylor_double_exp} and \eqref{eq:taylor_double_log}, we obtain the following (see appendix~\ref{appendix_taylor_pole} for the computations)

\begin{theorem}
\label{thm:pole}
Let $(\M, g)$ be a finite dimensional Riemannian manifold. Let $m\in \M$ and $v,w \in T_m\M$ sufficiently small. Then the output $u$ of one step of the pole ladder parallel transported back to $m$ is given by
\begin{equation}
    u = v + \frac{1}{12}\big( (\nabla_{w} R)(w,v)(5v -w) + (\nabla_{v} R)(w,v)(2v - w) \big) + O(5). \label{eq:taylor_pole}
\end{equation}
\end{theorem}
We notice that an elementary construction of pole ladder is more precise than that of Schild's ladder in the sense that there is no third order term $R(\cdot, \cdot)\cdot$. Using $n^{-1}$ to scale $w$ and $n^{-\alpha}$ to scale $v$ as in SL, the term $(\nabla_{w} R)(w,v)w$ is the dominant term when iterating this construction. Indeed, as it is linear in $v$, the scaling does not influence the convergence speed as one needs to multiply the result by $n^{\alpha}$ at the final rung of the ladder to recover the parallel transport of $v$. The other terms are multilinear in $v$, so if $\alpha >1$, there are small compared to $\frac{1}{n^3}(\nabla_{w} R)(w,v)w$. Summing $n$ terms of the form $\frac{1}{n^3}(\nabla_{w} R)(w,v)w$ we thus obtain a quadratic speed of convergence, as for Schild's ladder, for any $\alpha \geq 1$. We henceforth use $\alpha=1$. 

As in Sec.~\ref{sec:averaged}, one could think of applying pole ladder to linear combinations of $v,w$ to achieve a better performance, but this strategy is doomed as no linear combination can cancel out the dominant term $(\nabla_{w} R)(w,v)w$.
\begin{figure}[t]
    \centering
    \includegraphics[width=0.7\textwidth]{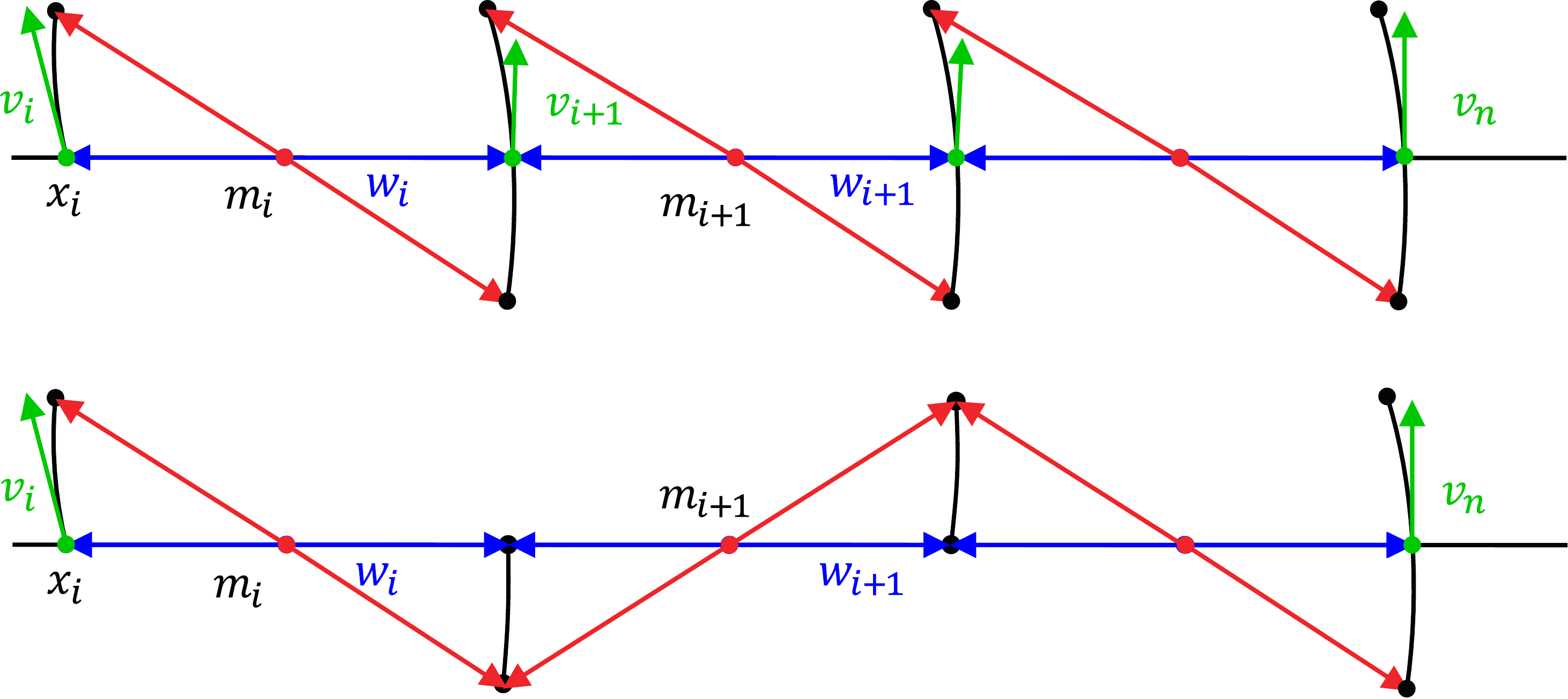}
    \caption{The pole ladder, consisting in iterations of elementary constructions (top). It can be simplified to lighten the computational burden of computing geodesics (bottom), only one log and one exp need to be computed from $m_i$ at each iteration.}
    \label{fig:pole_iterated}
\end{figure}
As in the previous section, define $u^w = \pole(m, w, v)$ the result of the PL construction (with the new notations), and the sequence
\begin{align}
    \label{eq:sequence_pole}
    v_0 &= v \nonumber \\
    v_{i+1} &= n \cdot \pole(m_i, \frac{w_i}{n}, \frac{v_i}{n}),
\end{align}
where $m_i = \gamma(\frac{2i + 1}{2n})=\exp_m \big( \frac{2i}{n} w \big)$, $w_i = \dot \gamma(\frac{2i + 1}{2n})= \frac{n}{2} \log_{m_i}(m_{i + 1}) = \Pi_m^{m_i} w$. Then it is straightforward to reproduce the proof of thm.~\ref{thm:convergence_schild} with \eqref{eq:taylor_pole} instead of \eqref{eq:taylor_schild} to show that $\exists \beta >0, \exists N \in \N, \forall n > N$,
\begin{equation}
    \label{eq:pole_convergence}
    \|v_n - \Pi^{m_n}_m v\| \leq \frac{\beta}{n^2},
\end{equation}
and $\beta$ can be bounded by the covariant derivative of the curvature.

This proves the convergence of the scheme with the same speed as SL. Note however that when iterating the scheme, unlike SL, the $v_i$'s and the geodesics that correspond to the rungs of the ladder need not being computed explicitly, except at the last step. This is shown Fig.\ref{fig:pole_iterated}, bottom row. This greatly reduces the computational burden of the PL over SL.

\subsection{Infinitesimal Schemes with Geodesic Approximations}
When geodesics are not available in closed form, we replace the exp map by a fourth-order numerical scheme (e.g. Runge-Kutta (RK)), and the log is obtained by gradient descent over the initial condition of exp. It turns out that only one step of the numerical scheme is sufficient to ensure convergence, and keeps the computational complexity reasonable. As only one step of the integration schemes is performed, we are no longer computing geodesic parallelograms, but infinitesimal ones, and thus refer to this variant as \textit{infinitesimal scheme}. We here detail the proof for the PL, but it could be applied to SL as well as the other variants.

More precisely, consider the geodesic equation written as a first order equation of two variables in a global chart $\Phi$ of $M$, that defines for any $p \in M$ a basis of $T_pM$, written $B^\Phi_x = \frac{\partial}{\partial x^i}\bigg|_x, \; i=1,\ldots,d$:

\begin{equation}
\label{eq:geo-flow}
  \left\{
    \begin{aligned}
      \dot x^k(t) &= v^k(t),\\
      \dot v^k(t) &= - \Gamma_{ij}^k v^i(t) v^j(t),
    \end{aligned}
  \right.
\end{equation}

where $\Gamma_{ij}^k$ are the Christoffel symbols, and we use Einstein summation convention. Let $rk: T\M \times \R_+ \rightarrow T\M$ be the map that performs one step of a fourth-order numerical scheme (e.g. RK4), i.e.\ it takes as input $(x(t), v(t)), h)$ and returns an approximation of $(x(t+h), v(t+h))$ when $(x,v)$ is solution of the system~\eqref{eq:geo-flow}. In our case, the step-size $h=\frac{1}{n}$ is used. By fourth-order, we mean that we have the following local \textit{truncation} error relative to the 2-norm of the global coordinate chart $\Phi$:
\begin{equation*}
    \| x(t+h) - rk_1(x(t), v(t), h) \|_2 \leq \tau_{1}h^5,
\end{equation*}
and global accumulated error after $n$ steps with step size $h=\frac{1}{n}$:
\begin{equation*}
    \| x(1) - \tilde x_n\|_2 \leq \frac{\tau_{1}}{n^4},
\end{equation*}
where by $rk_1$ we mean the projection on the first variable, and $\tau_{1}>0$ is a constant.

This means that any point on the geodesic is approximated with error $O(\frac{1}{n^4})$. As we are working in a compact domain, and all the norms are equivalent, the previous bounds can be expressed in Riemannian distance $d$:
\begin{align}
 d\big( x(t+h), rk_1(x(t), v(t), h) \big) &\leq \tau_{2}h^5    \label{eq:local-trunc} \\
 d\big( x(1), \tilde{x_n}\big) &\leq \frac{\tau_2}{n^4}.   \label{eq:global-trunc}
\end{align}

Furthermore, the inverse of $v\mapsto rk_1(x,v,h)$ can be computed for any $x$ by gradient descent, and we assume\footnote{see remark \ref{rk:convergence_log} thereafter about the validity of this hypothesis} that any desired accuracy can be reached. We write $\tilde v=rk_x^{-1}(y)$ for the optimal $v$ such that $d(rk(x,v, \frac{1}{n}), y) = o(\frac{1}{n^5})$, and assume that $\|rk_x^{-1}(y) - \log_x(y)\| \leq \frac{\tau_3}{n^4}$. Note that the step size does not appear explicitly in the notation $rk^{-1}$, but $h=\frac{1}{n}$ is always used. As in practice, $x\in \M, v,w\in T_x\M$ are given, the initial $w_0$ is tangent at $x$, and then for $i>0$, $w_i$ will be tangent at $m_i$ as in \eqref{eq:sequence_pole}, but $w_i$ maps $m_i$ to $m_{i+1}$ and this must correspond to one step of size $\frac{1}{n}$ in the discrete scheme, so there is a factor two that differs from the definition of $w_i$ in \eqref{eq:sequence_pole}. Now define also $\tilde m_i, \tilde w_i, \tilde v_i$, using $rk$ instead of exp and log, that is (see figure \ref{fig:infinitesimal_pole}):
\begin{align*}
    \label{eq:sequence_inf_pole}
    m_0 &= \exp_x(\frac{w}{2n})
    &\qquad
    z_0 &= \exp_x(\frac{v}{n})\\
    m_{i+1} &= \exp_{m_i}(\frac{w_i}{n})
    &\qquad
    z_{i+1} &= \exp_{m_i}(-\log_{m_i}(z_i))
\\
    \tilde m_0, \frac{\tilde w_0}{2} &= rk(x, \frac{w}{2}, \frac{1}{n})
    &\qquad
    \tilde z_0 &= rk_1(x, v, \frac{1}{n}) \\
    \tilde m_{i+1}, \tilde w_{i+1} &= rk(\tilde m_i, \tilde w_i, \frac{1}{n})
    &\qquad
    \tilde z_{i+1} &= rk_1(\tilde m_i, -rk^{-1}_{\tilde m_i}(\tilde z_i), \frac{1}{n} ).
\end{align*}
Finally let $x_n = \exp(w), v_n = (-1)^n n \log_{x_n}(z_n)$ and their approximations $\tilde x_n = rk(\tilde m_n, \frac{\tilde w_n}{2},\frac{1}{n})$ and $\tilde v_n = n \cdot rk^{-1}_{\tilde x_n}(\tilde z_n)$.

\begin{figure}
    \centering
    \includegraphics[width=0.5\textwidth]{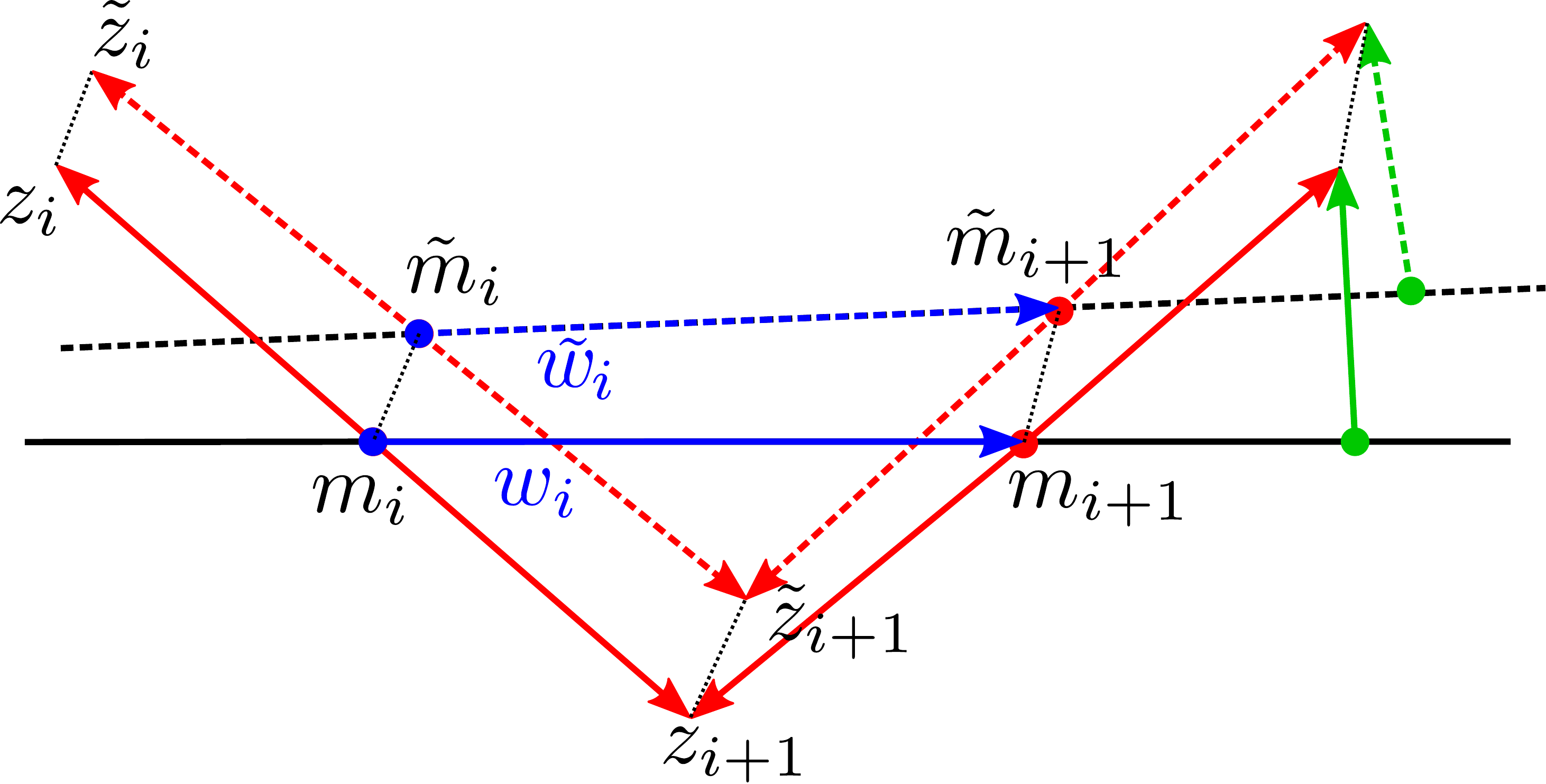}
    \caption{Representation of one iteration of the infinitesimal PL scheme. The exact geodesics and logs are plain lines, while their approximations with a numerical scheme are dashed.}
    \label{fig:infinitesimal_pole}
\end{figure}
We will prove that this approximate sequence converges to the true parallel transport of $v$:

\begin{theorem}
Let $(\tilde v_n)_n$ be the sequence defined above, corresponding to the result of the pole ladder with approximate geodesics computed by a fourth-order method in a global chart. Then
\begin{equation*}
    \| \Pi_x^{x_n} v - \tilde v_n\| = O(\frac{1}{n^2}).
\end{equation*}
\end{theorem}

\begin{proof}
It suffices to show that there exists $\beta'>0$ such that for $n$ large enough  $\| \tilde v_n - v_n\| \leq \frac{\beta'}{n^2}$. Indeed, by \eqref{eq:pole_convergence}, for $n$ large enough:
\begin{equation*}
    \label{eq:result}
    \| \Pi_x^{x_n} v - \tilde v_n\| \leq \| \Pi_x^{x_n} v - v_n\| + \|\tilde v_n - v_n\| \leq \frac{\beta + \beta'}{n^2}.
\end{equation*}
The approximations made when computing the geodesics with a numerical scheme accumulate at three steps: (1) the RK scheme compared to the true geodesic, this is controlled with the above hypotheses, (2) the distance between the results of the exp map when both the footpoint and the input vector vary a little, this is handled in lemma~\ref{lemma:exp-smoothness} below, and (3) the difference between the results of the log map when both the foot-point and the input vary a little, this is similar to (2) and is handled in lemma~\ref{lemma:log-continuity}. See figure~\ref{fig:lemma_exp_log} for a visual intuition of those lemmas.

\begin{lemma}
\label{lemma:exp-smoothness}
$\forall x, \tilde x \in M$ such that $x$ and $\tilde x$ are close enough, for all $v \in T_xM, \tilde v \in T_{\tilde v}M$ such that both $v$ and $\delta v = \Pi_{\tilde{x}}^x \tilde v - v$ are small enough, we have:
\begin{equation*}
    d \big( \exp_x(v), \exp_{\tilde x}(\tilde v) \big) \leq d(x, \tilde x) + \|  \Pi_{\tilde{x}}^x \tilde v - v \|.
\end{equation*}
\end{lemma}

\begin{proof}
Let $\delta x = \log_x(\tilde x)$, $x_v = \exp_x(v)$ and $\tilde{x}_{\tilde{v}} =  \exp_{\tilde x}(\tilde v)$. By the definition of the double exp and $\delta v$,  $\tilde{x}_{\tilde{v}} = \exp_x(h_x(\delta x, v + \delta v))$. Then by the definition of the neighboring log, we have
\begin{equation}
	\log_{x_v}(\tilde{x}_{\tilde{v}}) = \log_{x_v}(\exp_x(h_x(\delta x, v + \delta v)))= \Pi^{x_v}_x l_x(v, h_x(\delta_x, v + \delta v)).
\end{equation}
Using the Taylor approximations \eqref{eq:taylor_double_exp} and \eqref{eq:taylor_double_log} truncated at the order of $\|v\|$, we obtain
\begin{equation*}
	\Pi_{x_v}^x \log_{x_v}(\tilde{x}_{\tilde{v}}) = l_x(v, \delta x + v+ \delta v) = \delta x + \delta v.
\end{equation*}
So that the Riemannian distance $d$ between $x_v$ and $\tilde{x}_{\tilde{v}}$ is
\begin{align*}
	d = \| \log_{x_v}(\tilde{x}_{\tilde{v}}) \| \leq \|\delta x\| + \|\delta v\|,
\end{align*}
and the result follows.
\qed
\end{proof}

\begin{figure}
    \centering
    \includegraphics[width=0.4\textwidth]{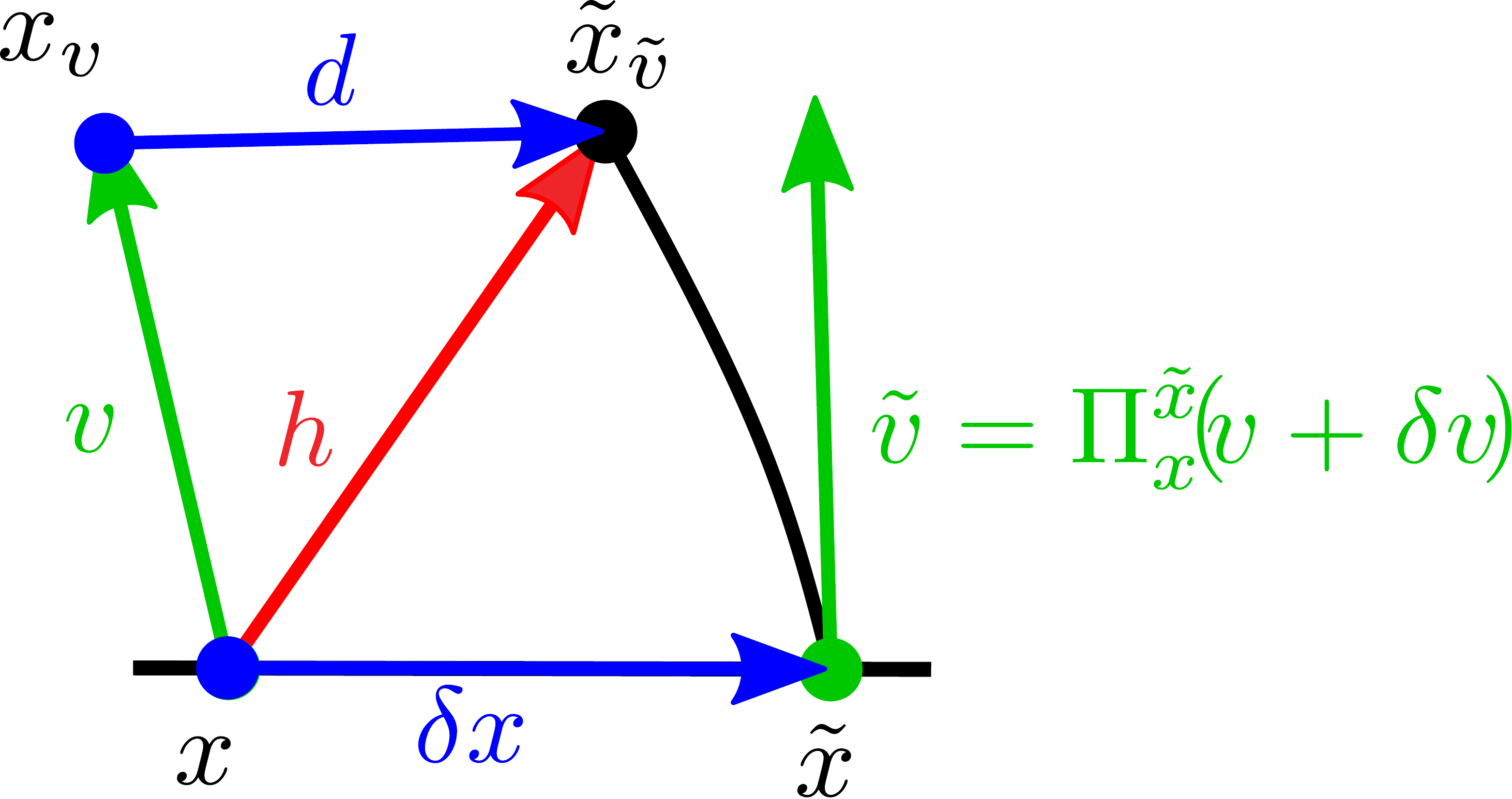}
    \caption{Visualization for the two lemmas: lemma~\ref{lemma:exp-smoothness} seeks to bound the norm of $d$ given $\delta x$ and $\delta v$ while lemma~\ref{lemma:log-continuity} bounds the norm of $\delta v$ given $d$ and $\delta x$. Here $h$ is the shorthand for $h_x(\delta x, v + \delta v)$ and $d$ for $\log_{x_v}(\tilde x_{\tilde v})$.}
    \label{fig:lemma_exp_log}
\end{figure}

\begin{lemma}
\label{lemma:log-continuity}
$\forall x, \tilde x, z \in M$ close enough to one another, we have in the metric norm
\begin{equation}
    \|\log_x(z) -\Pi_{\tilde{x}}^x \log_{\tilde x}(\tilde{z}) \| \leq d(x, \tilde{x}) + d(z, \tilde{z}).
\end{equation}
\end{lemma}

\begin{proof}
The proof is similar to that of lemma~\ref{lemma:exp-smoothness} except that this time it is the norm of $\|\delta v\|$ that needs to be bounded by $d(x_v, \tilde{x}_{\tilde{v}})$.
\qed
\end{proof}

Now, we first show that the sequence $\big(\delta_i = d(\tilde z_{i}, z_{i})\big)_i$ verifies an inductive relation, such that it is bounded by $\frac{1}{n^3}$, and then we will use lemma~\ref{lemma:log-continuity} to conclude. We first write
\begin{align*}
\| \log_{m_i}(z_i) - rk^{-1}_{\tilde m_i}(\tilde z_i)\| \leq \|\log_{m_i}(z_i) - \log_{\tilde m_i}(\tilde z_i)\| + \|\log_{\tilde m_i}(\tilde z_i) - rk^{-1}_{\tilde m_i}(\tilde z_i)\|.
\end{align*}
The second term on the r.h.s.\ corresponds to the approximation of the log by gradient descent and is bounded by hypothesis. For the first term, $d(m_i, \tilde{m_i})\leq \frac{\tau_2}{n^4}$ by hypothesis on the scheme \eqref{eq:global-trunc}. Suppose for a proof by induction on $i$ that $\delta_i = d(\tilde z_i, z_i) \leq \frac{\tau_2}{n}$ and $d(m_i, z_i)\leq \frac{\|w\| + 2\|v\|}{n}$. This is verified for $i=0$ and allows to apply lemma~\ref{lemma:log-continuity} for $n$ large enough, so
\begin{equation}
\| \log_{m_i}(z_i) - rk^{-1}_{\tilde m_i}(\tilde z_i)\| \leq d(\tilde m_i, m_i) + \delta_i + \frac{\tau_3}{n^4}.
\end{equation}
Furthermore, by lemma~\ref{lemma:exp-smoothness} applied to $v=-\log_{m_i}(z_i)$ and $\tilde{v} = -rk^{-1}_{\tilde m_i}(\tilde z_i)$, which are sufficiently small by the induction hypothesis,
\begin{align*}
\delta_{i+1} &= d(z_{i+1}, \tilde z_{i+1}) = d\big( rk_1(\tilde m_i, -rk^{-1}_{\tilde m_i}(\tilde z_i), \frac{1}{n}),  \exp_{m_i}(-\log_{m_i}(z_i))\big)\\
		&\leq d\big( rk_1(\tilde m_i, -rk^{-1}_{\tilde m_i}(\tilde z_i), \frac{1}{n}),  \exp_{\tilde m_i}(-rk^{-1}_{\tilde m_i}(\tilde z_i))\big) \\
        	&\qquad\quad + d\big( \exp_{\tilde m_i}(-rk^{-1}_{\tilde m_i}(\tilde z_i)), \exp_{m_i}(-\log_{m_i}(z_i))\big) \\
    		&\leq \frac{\tau_2}{n^5} +  d(m_i, \tilde m_i) + \| \log_{m_i}(z_i) - rk^{-1}_{\tilde m_i}(\tilde z_i)\|.
\end{align*}
And combining the two results, we obtain $\delta_{i+1} \leq \frac{(2n+1)\tau_2 + n\tau_3}{n^5} + \delta_i $, which completes the induction for $\delta_i$. For $d(z_i, m_i)$ we have:
\begin{equation}
d(z_{i+1}, m_{i+1}) = d(z_{i}, m_{i+1}) \leq d(x_i, z_i) + d(x_i, m_i) = \|\frac{v_i}{n}\| + \|\frac{w_i}{n}\|.
\end{equation}
In the section on SL, we proved that $\|v_i\|\leq 2\|v\|$ for $n$ large enough. This applies here as well, and the fact that $w_i$ is the parallel transport of $w$ completes the proof by induction.

By summing the terms for $i=0,\ldots,n-1$, and using $\delta_0\leq \frac{\tau_2}{n^5}$ by \eqref{eq:local-trunc}, we thus have $\delta_n \leq \frac{\tau_3 + 2\tau_2}{n^3} + \frac{n + 1}{n^5}\tau_2$.
We finally apply lemma~\ref{lemma:log-continuity} to the scaled $v_n, \tilde v_n$, as $x_n$ and $\tilde x_n$ are close enough:
\begin{align*}
 \frac{1}{n}\|v_n - \tilde v_n\| &= \|\log_{x_n}(z_n) - rk^{-1}_{\tilde x_n}(\tilde z_n) \| \\
 			&\leq \|\log_{x_n}(z_n) - \log_{\tilde x_n}(\tilde z_n)\| + \| \log_{\tilde x_n}(\tilde z_n) - rk^{-1}_{\tilde x_n}(\tilde z_n)\|\\
 			&\leq d(x_n, \tilde x_n) + \delta_n + \frac{\tau_2}{n^5}.
\end{align*}
So that for $n$ large enough $\|\tilde v_n - v_n\| \leq \frac{\beta'}{n^2}$ for some $\beta'>0$.
\qed
\end{proof}

\begin{remark}
\label{rk:convergence_log}
To justify the hypothesis on $rk^{-1}$, namely, $\|rk^{-1}_x(y) - \log_x(y)\|_2 \leq \frac{\tau}{n^5}$, we consider the problem \eqref{opt-P} which corresponds to an energy minimization. Working in a convex neighborhood, it admits a unique minimizer $v^*$:
\begin{equation}
\min \frac{1}{2}\|v\|_2^2 \qquad \text{s.t.}  \qquad \exp_x(v) = y. \tag{$\mathcal{P}$}\label{opt-P}
\end{equation}
The constraint can be written with Lagrange multipliers,
\begin{equation}
\min \frac{1}{2}\|v\|_2^2 + \lambda \|\exp_x(v) - y\|^2_2. \tag{$\mathcal{P'}$}\label{opt-P'}
\end{equation}
Now as the $\exp$ map is approximated by $rk_1$ at order 5 locally, writing for any $v$ and $h$ small enough, $\|\exp_x(hv) - y\|_2 \leq \|rk_1(x,v,h)-y\|_2 + \tau_1 h^5$,
\eqref{opt-P'} is equivalent to
\begin{equation}
\min \|rk_1(x,v,h) - y\|^2_2 + \frac{1}{2\lambda }\|v\|_2^2, \tag{$\mathcal{Q}$}\label{opt-Q}
\end{equation}
which is solved by gradient descent (GD) until a convergence tolerance $\epsilon \leq h^5$ is reached.
\end{remark}

\subsubsection{Numerical Simulations}
It is not relevant to compare the PL with SL on spheres or SPD matrices
as in the previous section, as theses spaces are symmetric and thus the PL is exact \cite{guigui_symmetric_2019}. We therefore focus on the Lie group of isometries of $\R^3$, endowed with a left-invariant metric $g$ with the diagonal matrix at identity:
\begin{equation}
   G = \mathrm{diag}(1,1,1,\beta,1,1),
\end{equation}
where the first three coordinates correspond to the basis of the Lie algebra of the group of rotations, while the last three correspond to the translation part, and $\beta>0$ is a coefficient of anisotropy. These metrics were considered in \cite{zefran_generation_1998} in relation with kinematics, and visualisation of the geodesics was provided, but no result on the curvature was given. Following \cite{milnor_curvatures_1976}, we compute explicitly the covariant derivative of the curvature and deduce (proof given in appendix~\ref{appendix:sen})

\begin{lemma}
\label{lemma:sen}
$(SE(3),g)$ is locally symmetric, i.e.\ $\nabla R=0$, if and only if $\beta = 1$.
\end{lemma}

This allows to observe the impact of $\nabla R$ on the convergence of the PL. 
In the case where $\beta=1$, the Riemannian manifold $(SE(3),g)$ corresponds to the direct product $SO(3)\times \R^3$ with the left-invariant product metric formed by the canonical bi-invariant metric on the group of rotations $SO(3)$ and the canonical inner-product of $\R^3$. 
Therefore the geodesics can be computed in closed form.
Note that the left-invariance refers to the group law which encompasses a semi-direct action of the rotations on the translations. 
When $\beta \not = 1$ however, geodesics are no longer available in closed form and the infinitesimal scheme is used, with the geodesic equations computed numerically (detailed in appendix~\ref{appendix:sen}). In this case, we compute the pole ladder with an increasing number of steps, and then use the most accurate computation as reference to measure the empirical error.

\begin{figure}
    \begin{minipage}{.49\textwidth}
        \centering
        \includegraphics[width=0.9\textwidth]{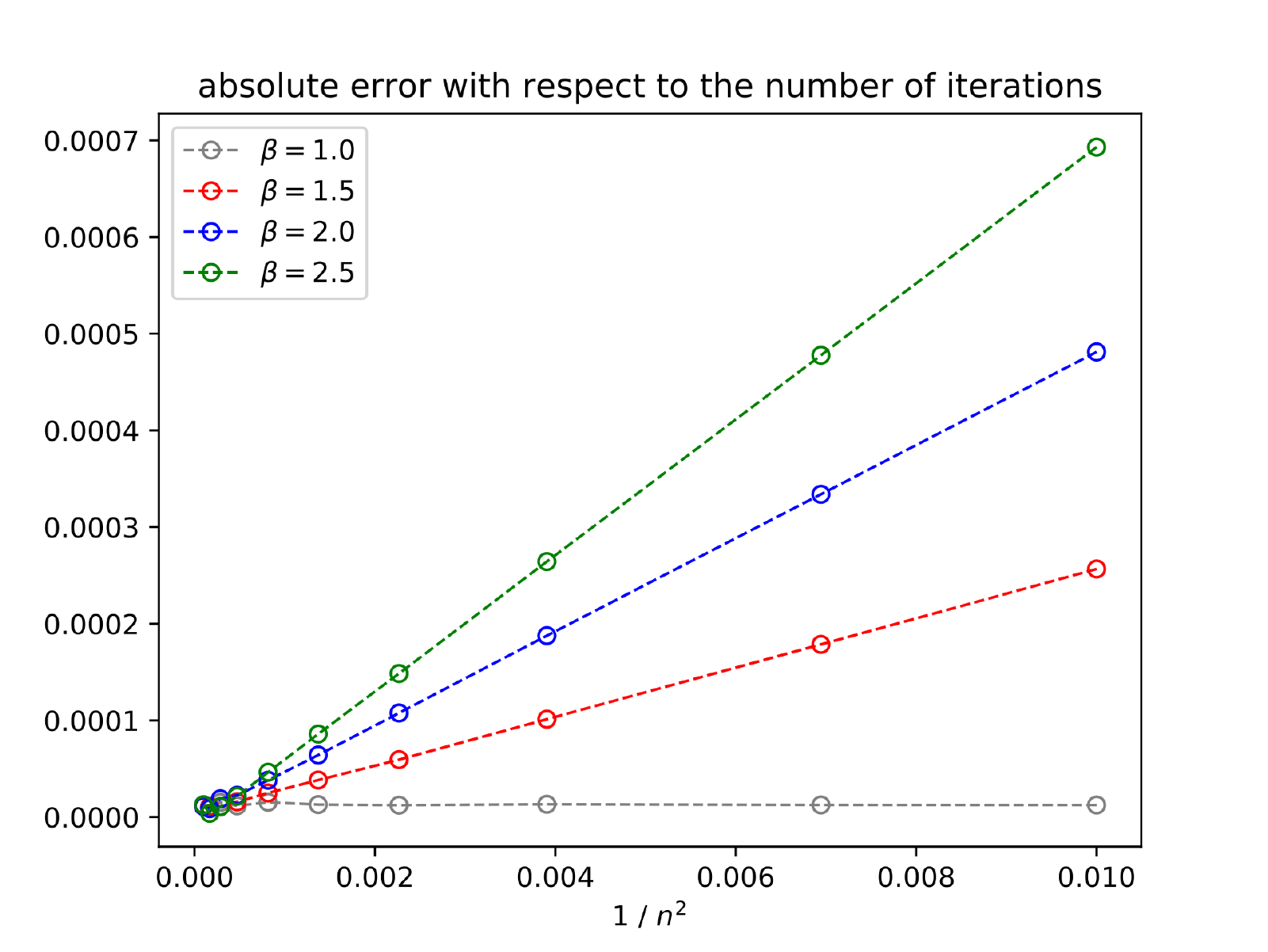}
    \end{minipage}
    \begin{minipage}{.5\textwidth}
        \centering
        \includegraphics[width=0.9\textwidth]{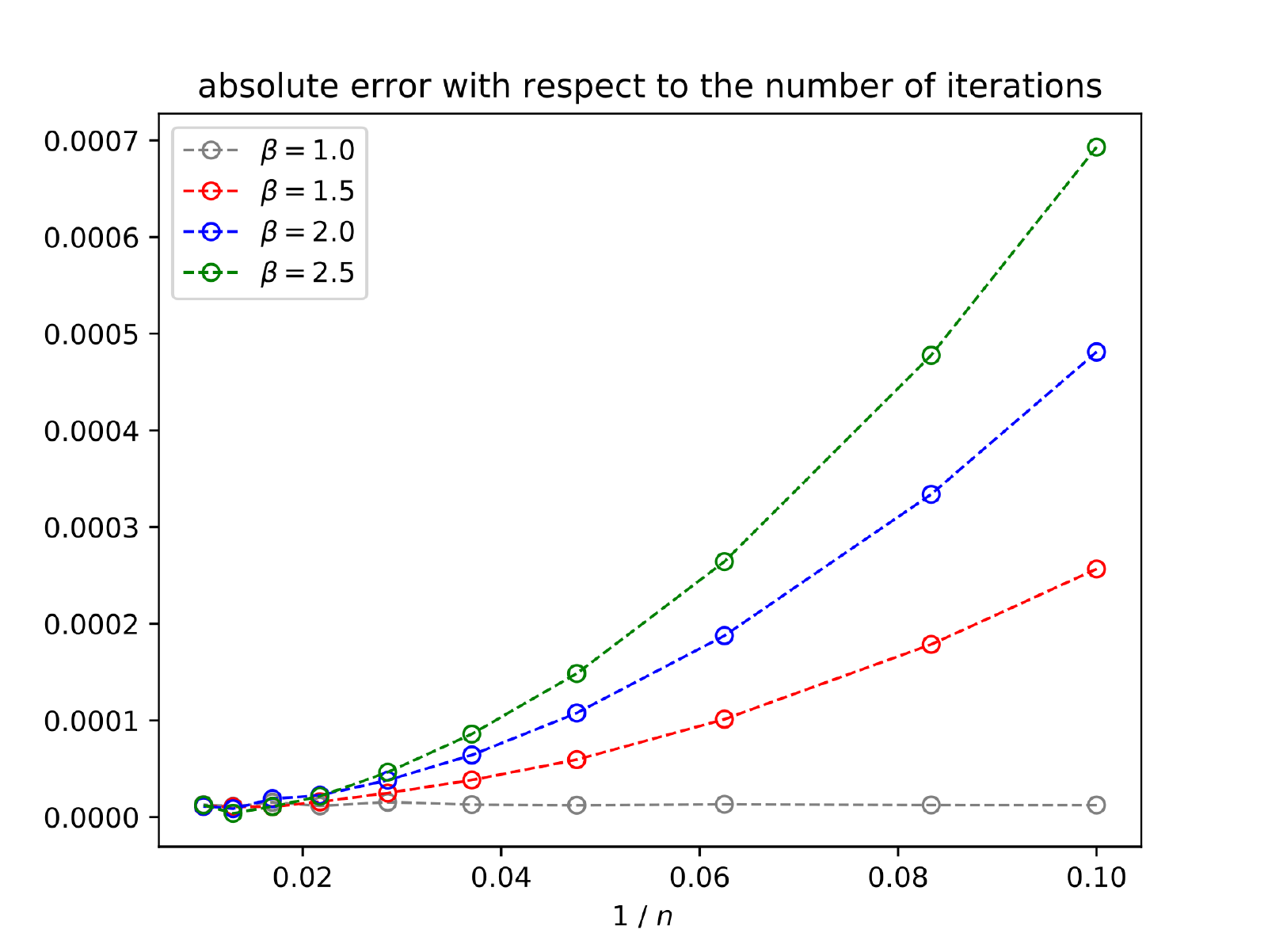}
    \end{minipage}
    \caption{Error of the parallel transport of $v$ along the geodesic with initial velocity $w$ where $v$ and $w$ are orthonormal basis vectors (they stay the same for all values of $\beta$). In accordance with our main result, a quadratic speed of convergence is reached, and the speed depends on the asymmetry of the space, itself induced by the anisotropy $\beta$ of the metric. Left : Absolute error w.r.t\ $\frac{1}{n^2}$: these are straight lines with slopes growing with $\beta$. Right: Absolute error w.r.t.\ $\frac{1}{n}$ showing the quadratic convergence.}
    \label{fig:pole_sen}
\end{figure}

Results are displayed on Fig.~\ref{fig:pole_sen}. Firstly, we observe very precisely the quadratic convergence of the infinitesimal scheme, as straight lines are obtained when plotting the error against $\frac{1}{n^2}$. Secondly, we see how the slope varies with $\beta$: accordingly with our results, it cancels for $\beta=1$ and increases as $\beta$ grows.

Finally, for completeness, we compare Schild's ladder and the pole ladder in this context. We cannot choose two basis vectors $v,w$ (that don't change when $\beta$ changes) such that $R(w,v)v \neq 0$ when $\beta=1$ and $\nabla_w R(w,v)w \neq 0$ when $\beta \neq 1$. Indeed, the former condition implies that $v,w$ are infinitesimal rotations, which implies $\forall \beta, \nabla_w R(w,v)w=0$.
Therefore, we choose $v,w$ such that the latter condition is verified, so that the SL error is of order four in our example (but it is not exact), and it cannot be distinguished from PL, but this is only a particular case. The results are shown on Fig.~\ref{fig:schild_pole_anisotropic}. Note that due to the larger number of operations required for SL, our implementation is less stable and diverges when $n$ grows too much ($\sim 50$). As expected when $\beta>1$, the speeds of convergence are of the same order for PL and SL, but the multiplicative constant is smaller for PL.

We also compare ladder methods to the Fanning Scheme, and as expected the quadratic speed of convergence reached by ladder methods yields a far better accuracy even for small $n$. We now compare the infinitesimal SL, PL and the FS in terms of computational cost.

\begin{figure}
    \begin{minipage}{.49\textwidth}
        \centering
        \includegraphics[width=0.9\textwidth]{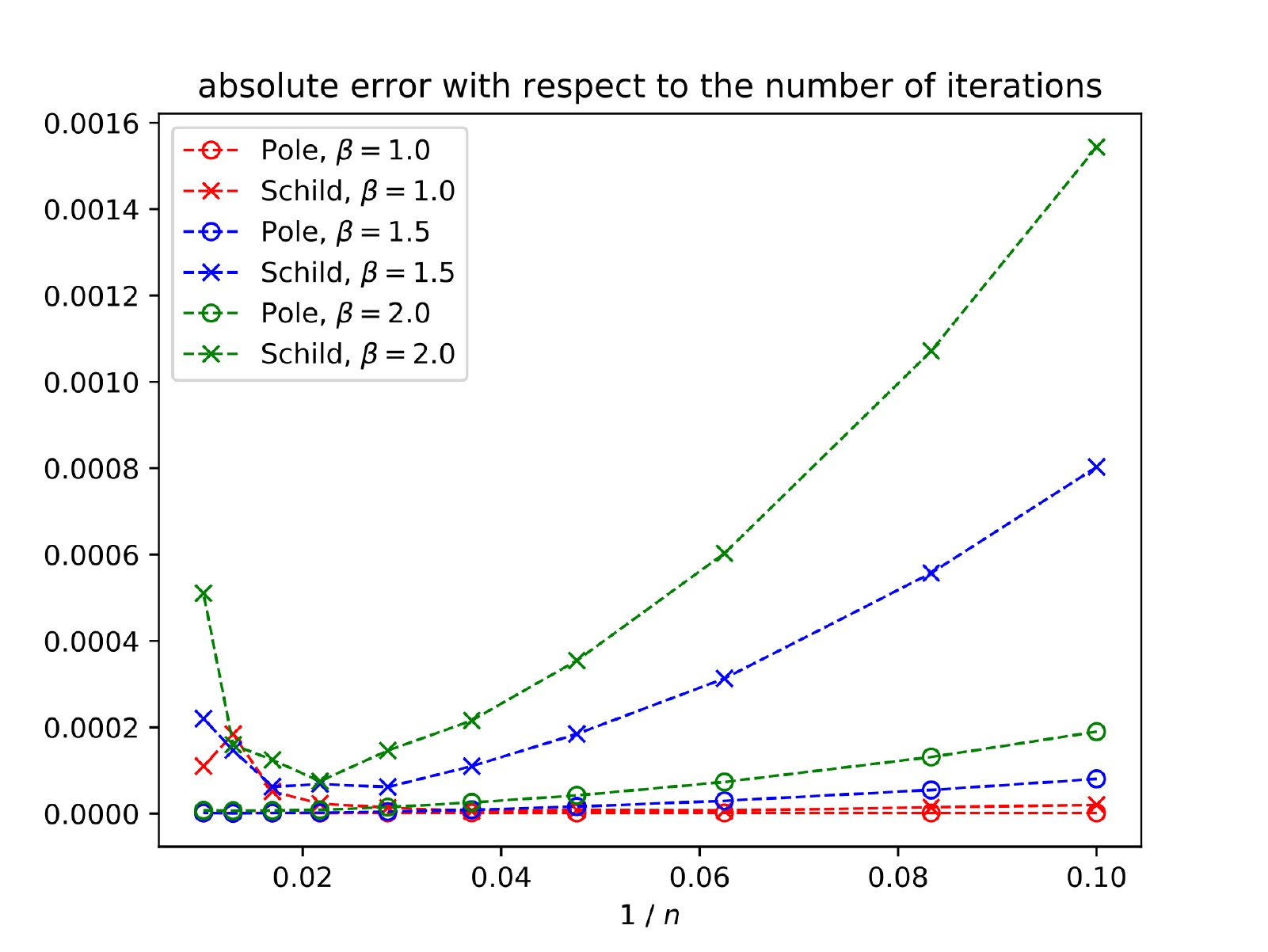}
    \end{minipage}
    \begin{minipage}{.5\textwidth}
        \centering
        \includegraphics[width=0.9\textwidth]{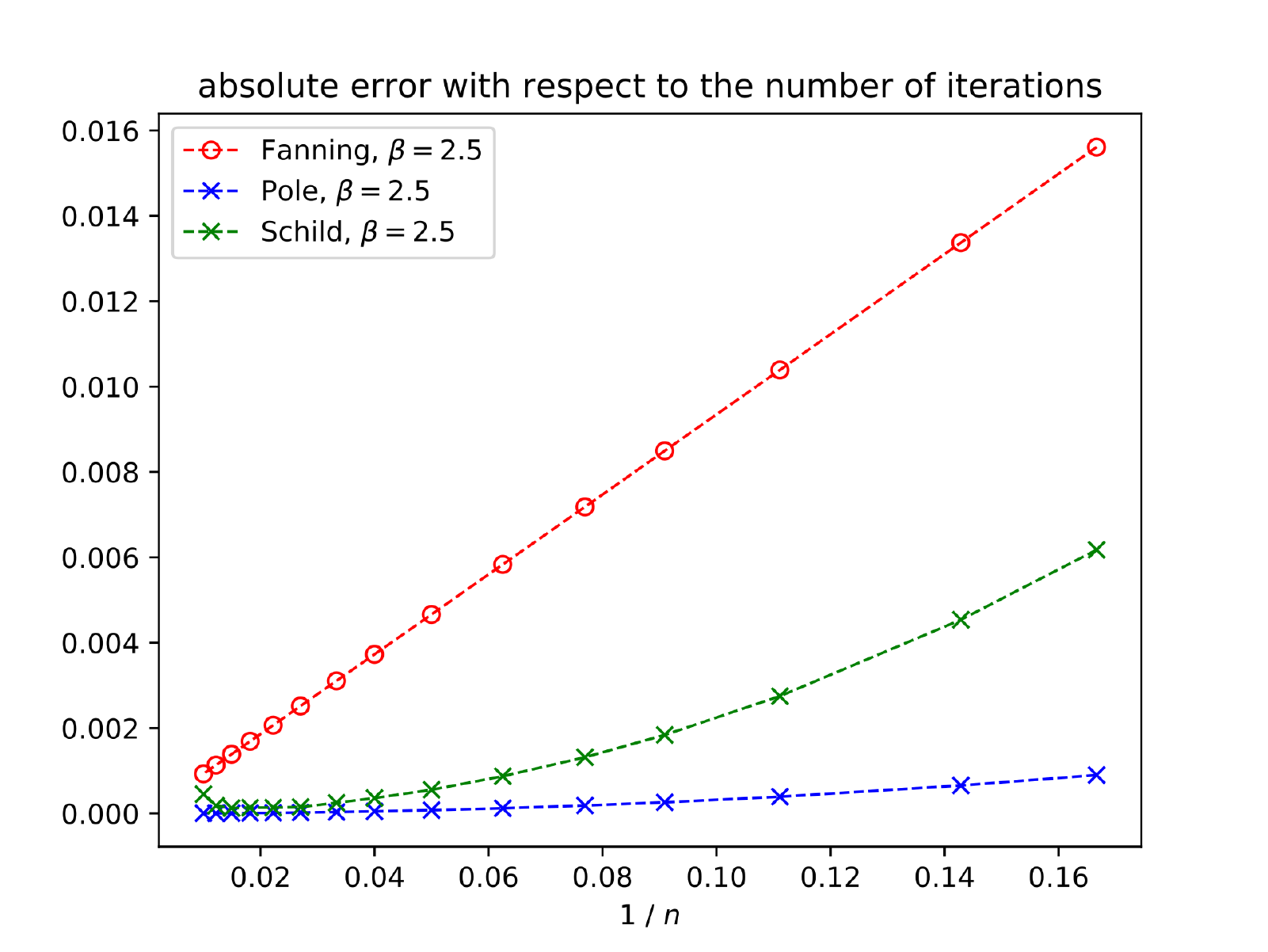}
    \end{minipage}
    \caption{left: Comparison of the speed of convergence of SL and PL for two orthonormal basis vectors of $T_xSE(3)$ depending on the coefficient of anisotropy $\beta$ of the metric. PL converges faster and is much more stable right: Comparison of SL, the PL and the FS in the same setting with fixed $\beta$. (Note that absolute and relative errors are the same here as $v$ has unit norm.)}
    \label{fig:schild_pole_anisotropic}
\end{figure}

\subsubsection{Remarks on Complexity}
\label{sec:complexity}
At initialization, ladder methods require to compute $x_v$, with one call to the numerical scheme $rk$ (e.g Runge-Kutta). Then the main geodesic needs to be computed, for SL and FS it requires $n$ calls to $rk$. For PL, we only take a half step at initialization, to compute the first midpoint, then compute all the $m_i$s so that one final call to $rk$ is necessary to compute the final point of the geodesic $x_n$, totalling $n+1$ calls. Then at every iteration, considering given $m_i, z_i$, one needs to compute an inversion of $rk$, and then shoot with twice (or minus for PL) the result. In practice the inversion of $rk$ by gradient descent (i.e. shooting) converges in about five or six iterations, each requiring one call to $rk$. This operation thus requires less than ten calls to $rk$. Additionally, for SL only, the midpoint needs to be computed by one inversion and one call to $rk$, thus adding ten calls to the total. At the final step, another inversion needs to be computed, adding less than ten calls. Therefore, SL requires $21n + 11$ calls to $rk$, the PL $11(n+1)$. In contrast, the FS only requires to compute one (or two) perturbed geodesic and finite differences instead of the approximation of the log at every step. It therefore require $3n$ calls to $rk$.

Moreover as the FS is intrinsically a first-order scheme, a second-order $rk$ step is sufficient to guarantee the convergence, thus requiring only two calls to the geodesic equation ---the most significantly expensive computation at each iteration--- while ladder methods require a fourth-order scheme, which is twice as expensive. However, this additional cost is quickly balanced as only $O(\sqrt{n})$ steps are needed to reach a given accuracy, where $O(n)$ are required for the FS. Comparing the values of $n\mapsto 4 * \frac{11}{\sqrt{n + 1}}$ and $n \mapsto 2*\frac{3}{n}$, we conclude that PL is the cheapest option as soon as we want to achieve a relative accuracy better than $2.10^{-2}$. 
Note that this doesn't take into account the constants $\beta$, that may differ, but the previous numerical simulations show that these estimates are valid: a regression on the cost for the PL gives $y=41n + 45$. Finally, in the experiment of Fig.~\ref{fig:schild_pole_anisotropic}, the PL with $n=6$ yields a precision of $8.10^{-4}$ for a cost of $304$ calls to the equations, while $n=250$ is required to reach that precision with the FS, for a cost of $1500$ calls. For a similar computational budget, the PL reaches a precision $\sim 2.10^{-5}$ with $n=37$.

\section{Conclusion}

In this paper, we jointly analysed the behaviour of ladder methods to compute parallel transport. We first gave a Taylor expansion of one step of Schild's Ladder. Then, we showed that when scaling the vector to transport by $\frac{1}{n^2}$, a quadratic speed of convergence is reached. Our numerical experiments illustrate that this bound is indeed reached. In the same framework, we bridged the gap between the Fanning Scheme and Schild's Ladder, shedding light on how SL could be turned into a second-order method while the FS cannot.

For manifolds with no closed-form geodesics, we introduced the infinitesimal ladder schemes and showed that the PL converges with the same order as its counterpart with exact exp and log maps. The same exercise can be realized with SL. 
Numerical experiments were performed on $SE(3)$ with anisotropic metric, and allowed to observe the role of $\nabla R$ in the convergence of the PL in a non-symmetric space. This result corroborates our theoretical developments, and shows that the bounds on the speed of convergence are reached.




Our last comparison of SL and PL shows that, although more popular, SL is far less appealing that the PL as (i) it is more expensive to compute, (ii) it converges slower, (iii) it is less stable when using approximate geodesics and (iv) it is not exact in symmetric spaces. Pole ladder is also more appealing than the FS because of its quadratic speed of convergence, which allows to reach mild convergence tolerance at a much lower overall cost despite the higher complexity of each step.

The ladder methods are restricted to transporting along geodesics, but this not a major drawback as this is common to other methods, and one may approximate any curve by a piecewise geodesic curve.

In our work on infinitesimal schemes, we only used basic integration of ODEs in charts, while there is a wide literature on numerical methods on manifolds. In future work, it would be very interesting to consider specifically adapted RK schemes on Lie groups and homogeneous spaces \cite{munthe-kaas_integrators_2016,munthe-kaas_numerical_1997}, or symplectic integrators \cite{hairer_symplectic_2002, dedieu_symplectic_2005} in order to reduce the computational burden and to improve the stability of the scheme. More precisely, the computations of the PL may be hindered by the approximation of the log, while in fact, only a geodesic symmetry $y\mapsto \exp_x(-\log_x(y))$ is necessary and may be computed more accurately with a symmetric scheme.

\begin{acknowledgements}
The authors have received funding from the European Research Council (ERC) under the European Union’s Horizon 2020 research and innovation program (grant agreement G-Statistics No 786854). The authors warmly thank Yann Thanwerdas and Paul Balondrade for insightful discussions and careful proofreading of this manuscript.
\end{acknowledgements}

\bibliographystyle{spmpsci}      
\bibliography{Paper_Ladders.bib}

\newpage
\appendix
\section{Computation of the expansion of Schild's ladder}
\label{appendix_taylor_schild}
The details of the Taylor approximation for SL are given below at the fourth order, and a lemma to bound the fourth order terms is proved. First we combine \eqref{eq:taylor_double_exp} and \eqref{eq:taylor_double_log} to compute an approximation of $a = \log_x(m)$ where $m$ is the midpoint between $x_v$ and $x_w$. That is, $a = h_x(v, \frac{1}{2} b)$ where $b = l_x(v, w)$:

\begin{align*}
    2a &= v + b +\frac{1}{6} R(b,v)(v + b) + \frac{1}{24}\big( (\nabla_v R)(b,v)(\frac{5}{2}b + 2v) + (\nabla_{b} R)(\frac{1}{2}b,v)(v + b) \big) + O(5) \\
       &= v + w - v +\frac{1}{6} R(v,w)(2w - v) + \frac{1}{24}\Big( (\nabla_v R)(v,w)(3w - 2v) \nonumber\\
        &\quad+ (\nabla_{w} R)(v,w)(2 w - v) \Big) +\frac{1}{6} R(\frac{1}{2}(w-v),v)(v + w-v) \nonumber\\
        &\quad+ \frac{1}{24}\big( (\nabla_v R)(w-v,v)(\frac{5}{2}(w-v) + 2v)  + (\nabla_{w-v} R)(\frac{1}{2}(w-v),v)(v + w-v) \big) + O(5)\\
       &= w + \frac{1}{6} R(v,w)(2w - v - w) + \frac{1}{24}\Big( (\nabla_v R)(v,w)(3w - 2v -2v - \frac{5}{2}(w-v))\nonumber\\
        &\quad + (\nabla_{w} R)(v,w)(2 w - v) + (\nabla_{w-v} R)(v,w)(- \frac{1}{2} w) \Big) + O(5).
\end{align*}
Therefore,
\begin{align}
    2a &= w + v + \frac{1}{6} R(v, w)(w - v)\nonumber\\
        &\quad +  \frac{1}{24}\Big( (\nabla_v R)(v,w)(w - \frac{3}{2}v)+ (\nabla_w R)(w,v)(v - \frac{3}{2}w)\Big) + O(5). \label{eq:2a_bis}
\end{align}
Now we compute $u = l_x(w, 2a)$:
\begin{align*}
    u &= 2a - w + \frac{1}{6} R(2a, w)(w - 4a) \nonumber\\
        &\quad + \frac{1}{24}\Big( (\nabla_w R)(w, 2a)(3*2a - 2w) + (\nabla_{2a} R)(w, 2a)(2* 2a - w) \Big) + O(5) \\
      &= w + v + \frac{1}{6} R(v, w)(w - v) \nonumber \\
        &\quad +  \frac{1}{24}\Big( (\nabla_v R)(v,w)(w - \frac{3}{2}v) + (\nabla_w R)(w,v)(v - \frac{3}{2}w)\Big)\nonumber \\
        &\quad - w + \frac{1}{6} R\big( w + v, w\big)(w - 2(w + v)) + \frac{1}{24}\Big( (\nabla_w R)(w, w + v)(3w +3v - 2w)\nonumber\\
        &\quad + (\nabla_{w+v} R)(w, w+v)(2w +2v - w) \Big) + O(5) \\
      &= v + \frac{1}{6} R(v, w) (w - v - w - 2v) \nonumber \\
        &\quad +  \frac{1}{24}\Big( (\nabla_v R)(v,w)(w - \frac{3}{2}v -2v -w) + (\nabla_w R)(v, w)(v - \frac{3}{2}w - w - 3v -2v -w)\Big). 
\end{align*}
Thus
\begin{align}
    u &= v - \frac{1}{2} R(v, w)v \nonumber\\
        &\quad + \frac{1}{24}\Big( (\nabla_v R)(v,w)(-\frac{7}{2}v) + (\nabla_w R)(v,w)(-4v - \frac{7}{2}w)\Big) + O(5). \label{eq:schild_bis}
\end{align}
We deduce the following
\begin{lemma}
\label{lemma:r4}
    With the previous notations, at $x$ in a compact set $K$, $u$ can be written $u = v + \frac{1}{2} R(w, v)v + r_4(v,w)$ and there exists $C>0$ such that $\exists \eta>0, \forall v,w \in B_0(\delta) \subset T_x\M, \forall |s|,|t|\leq \eta$, $r_4$ verifies:
    \begin{equation}
        \|r_4(sv,tw)\| \leq C (s^3t + s^2 t^2 + s t^3) \delta^4.
    \end{equation}
If furthermore $|s| \leq |t|$, then this reduces to $\|r_4(sv,tw)\| \leq C s t^3 \delta^4$. Moreover, $C$ can be bounded by bounds on the covariant derivatives of the curvature tensor.
\end{lemma}

\begin{proof}
Let $t' = \min(|t|, |s|)$. By \eqref{eq:schild_bis} we get
\[r_4(sv, tw) = \frac{st}{24}\Big( (\nabla_v R)(w,sv)(\frac{7}{2}sv) + (\nabla_w R)(tw,v)(4sv + \frac{7t}{2}w)\Big) + O(5). \]
Each term of the form $\nabla_\cdot R(tw,sv)\cdot$ can be bounded, for example:
\begin{equation}
    \|st(\nabla_v R)(w,sv)(\frac{7}{2}sv)\| \leq s^3 t \|\nabla R\|_\infty \|w\| \|v\|^3,
\end{equation}
where the infimum norm on $\nabla R$ is taken on the compact set $K$ (thus it exists and it is finite). Similarly, as $\|w\| \leq \delta$ and $\|v\| \leq \delta$
\begin{equation}
    \|s(\nabla_v R)(tw,sv)(\frac{7}{2}sv)\| \leq s^3 t \|\nabla R\|_\infty \delta^3 .
\end{equation}
By dealing in a similar fashion with the two other terms, we obtain:
\begin{equation}
    \|r_4(sv,tw)\| \leq \|\nabla R\|_\infty (s^3t + s^2 t^2 + s t^3) \delta^4 + O(5),
\end{equation}
where $O(5)$ is a combination of terms of the form $q_5(sa, tb)$, which are homogeneous polynomials of degree at least five and variables in the ball of radius $\delta$, hence they can be bounded above by $C_1 (s^3t + s^2 t^2 + s t^3) \delta^4$ for some $C_1$. The result follows with $C= \|\nabla R\|_\infty + C_1$.
\qed
\end{proof}

\section{Computation of the expansion of the pole ladder}
\label{appendix_taylor_pole}
As in the previous appendix, we give here the computations for the fourth-order Taylor approximation of the PL construction. Recall (Fig.~\ref{fig:pole_construction}) that $v,w \in T_m\M$ where $m=\gamma(\frac{1}{2})$ is the midpoint between $x=\gamma(0)$ and $\gamma(1)$. The result of the construction parallel transported back to $m$ is $u$ such that:
\begin{align*}
    -u &= l_m(w, -h_m(-w, v)) \\
       &= w + h_m(-w, v) + \frac{1}{6}R(w, -h_m(-w, v))(-2h_m(-w, v) - w) \\
        &\quad + \frac{1}{24}\big( (\nabla_w R)(w,-h_m(-w, v))(-3h_m(-w, v) - 2w) \\
        &\quad+ (\nabla_{-h_m(-w, v)} R)(w,-h_m(-w, v))(-2h_m(-w, v) - w) \big) + O(5) . \\
\end{align*}
Now we plug \eqref{eq:taylor_double_exp}, but it reduces to $-h_m(-w, v)=w-v + O(3)$ when it appears in a term including curvature (because we restrict to fourth order terms overall). Hence
\begin{align*}
    -u &= -( - w + v +\frac{1}{6} R(v,-w)(-w + 2v) + \frac{1}{24}\big( (\nabla_{-w} R)(v,-w)(5v - 2w)\\
          &\quad + (\nabla_{v} R)(v,-w)(-w + 2v) \big) + q_5(w,v)) + \frac{1}{6}R(w, w-v)(2(w-v) - w) \\
          &\quad + \frac{1}{24}\big( (\nabla_w R)(w,w-v)(3(w - v) - 2w) + (\nabla_{w-v} R)(w,w-v)(2(w - v) - w) \big) + O(5) \\
       &= -v + \frac{1}{6} R(v,w)(-w + 2v) + \frac{1}{6}R(w, -v)(w - 2v)\\
          &\quad + \frac{1}{24}\big( (\nabla_{w} R)(v,-w)(5v - 2w) + (\nabla_{v} R)(v,w)(-w + 2v) \big)\\
          &\quad + \frac{1}{24}\big( (\nabla_w R)(v,w)(w - 3v) + (\nabla_{w-v} R)(v,w)(w - 2v) \big) + O(5)\\
       &= -v + \frac{1}{12}\big( (\nabla_{w} R)(v,w)(5v -w) + (\nabla_{v} R)(v,w)(2v - w) \big) +O(5) .
\end{align*}

\section{Geometry of the sphere}
\label{appendix:sphere}
The unit-sphere is defined as $S^2=\{x \in \R^3, \|x\|=1\}$. With the canonical scalar product $<\cdot, \cdot>$ of $\R^3$, it is a Riemannian manifold of constant curvature whose geodesics are great circles. The tangent space of any $x \in S^2$ is the set of vectors orthogonal to $x$: $T_xS^2 = \{v \in \R^3, <x,v>=0\}$. 
We have $\forall x,y \in S^2, \forall v,w \in T_xS^2$:
\begin{align*}
    \exp_x(w) &= \cos(\|w\|)x + \sin(\|w\|)\frac{w}{\|w\|},\\
    \log_x(y) &= \arccos(<y,x>)\frac{y - <y,x>x}{\|y - <y,x>x\|},\\
    \Pi_x^{x_w}v &= <v,\frac{w}{\|w\|}> \left(- \sin(\|w\|)x + \cos(\|w\|)\frac{w}{\|w\|}\right) + \left(v - <v,\frac{w}{\|w\|}>\frac{w}{\|w\|} \right).
\end{align*}
The expression for parallel transport of $v$ means that the orthogonal projection of $v$ on $\{w\}^{\perp}$ is preserved, while the orthogonal projection on $\R w$ is rotated by an angle $\|w\|$ in the $(x,w)$-plane.

\section{Affine-Invariant geometry of SPD matrices}
\label{appendix:spd}
The cone of $3 \times 3$ symmetric positive-definite matrices is defined as
\begin{equation*}
    SPD(3)=\{\Sigma \in \R^{3\times 3}, \Sigma^T=\Sigma, \;\; \forall x \in R^3\setminus\{0\},\;\; x^T\Sigma x > 0\}.
\end{equation*}
The tangent space of any $\Sigma \in SPD(3)$ is the set symmetric matrices: $T_\Sigma SPD(3) = Sym(3)$. The affine-invariant (AI) metric is defined at any $\Sigma \in SPD(3)$, for all $V,W \in Sym(3)$ using the matrix trace $tr$ by
\begin{equation*}
    g_\Sigma(V,W) = tr(\Sigma^{-1}V \Sigma^{-1}W).
\end{equation*}
We have $\forall \Sigma, \Sigma_1,\Sigma_2 \in SPD(3), \forall W \in Sym(3)$
\begin{align*}
    \exp_\Sigma(W) &= \Sigma^{\frac{1}{2}}\exp(\Sigma^{-\frac{1}{2}}W\Sigma^{-\frac{1}{2}})\Sigma^{\frac{1}{2}},\\
    \log_{\Sigma_1}(\Sigma_2) &= \Sigma_1^{\frac{1}{2}}\log(\Sigma_1^{-\frac{1}{2}}\Sigma_2 \Sigma_1^{-\frac{1}{2}})\Sigma_1^{\frac{1}{2}},
\end{align*}
where when not indexed, $\exp$ and $\log$ refer to the matrix operators. Finally, let 
\[P_t = \Sigma^{\frac{1}{2}}\exp(\frac{t}{2}\Sigma^{-\frac{1}{2}}W\Sigma^{-\frac{1}{2}})\Sigma^{-\frac{1}{2}}.\]
The parallel transport from $\Sigma$ along the geodesic with initial velocity $W\in Sym(3)$ of $V\in Sym(3)$ a time $t$ is (\cite{yair_parallel_2019})
\begin{equation*}
    \Pi_{0,W}^t V = P_t V P_t^T.
\end{equation*}

\section{Left-invariant metric on SE(3)}
\label{appendix:sen}
In this appendix, we describe the geometry of the Lie group  of isometries of $\R^3$, $SE(3)$, endowed with a left-invariant metric $g$. We prove lemma~\ref{lemma:sen}, that gives a necessary and sufficient condition for $(SE(3), g)$ to be a Riemannian locally symmetric space. We also give the details for the computation of the geodesics with numerical integration schemes. Those results and computations are valid in any dimension $d \geq 2$, but we only detail them for $d=3$ to keep them tractable.
For the details on this section, we refer the reader to \cite{milnor_curvatures_1976, gallier_differential_2020, kolev_lie_2004}.

$SE(3)$, is the semi-direct product of the group of three-dimensional rotations $SO(3)$ with $\R^3$, i.e. the group multiplicative law for $R,R' \in SO(3), t,t' \in \R^3$ is given by
\begin{equation*}
    (R,t)\cdot (R',t') = (RR', t + Rt').
\end{equation*}
It can be seen as a subgroup of $GL(4)$ and represented by homogeneous coordinates:
\begin{equation*}
    (R,t) = \begin{pmatrix} R & t \\ 0 & 1 \end{pmatrix},
\end{equation*}
and all group operations then correspond to the matrix operations. Let the metric matrix at the identity be diagonal: $G=\mathrm{diag}(1,1,1,\beta,1,1)$ for some $\beta>0$, the anisotropy parameter. We write $<\cdot, \cdot>$ for the associated inner-product at the identity. An orthonormal basis of the Lie algebra $\se(3)$ is
\begin{align*}
    e_1 = \frac{1}{\sqrt{2}}\begin{pmatrix} 0 & 0 & 0 & 0 \\
            0 & 0 & -1 & 0\\ 0 & 1 & 0 & 0 \\ 0 & 0 & 0 & 0 \end{pmatrix}
    &\qquad&
    e_2 = \frac{1}{\sqrt{2}}\begin{pmatrix} 0 & 0 & 1 & 0 \\
        0 & 0 & 0 & 0\\ -1 & 0 & 0 & 0 \\ 0 & 0 & 0 & 0 \end{pmatrix}
    &\qquad&
    e_3 = \frac{1}{\sqrt{2}}\begin{pmatrix} 0 & -1 & 0 & 0 \\
        1 & 0 & 0 & 0\\ 0 & 0 & 0 & 0 \\ 0 & 0 & 0 & 0 \end{pmatrix}\\
    e_4 = \frac{1}{\sqrt{\beta}}\begin{pmatrix} 0 & 0 & 0 & 1 \\
        0 & 0 & 0 & 0\\ 0 & 0 & 0 & 0 \\ 0 & 0 & 0 & 0 \end{pmatrix}
    &\qquad&
    e_5 = \begin{pmatrix} 0 & 0 & 0 & 0 \\
        0 & 0 & 0 & 1\\ 0 & 0 & 0 & 0 \\ 0 & 0 & 0 & 0 \end{pmatrix}
    &\qquad&
    e_6 = \begin{pmatrix} 0 & 0 & 0 & 0 \\
        0 & 0 & 0 & 0\\ 0 & 0 & 0 & 1 \\ 0 & 0 & 0 & 0 \end{pmatrix} .
\end{align*}
Define the corresponding structure constants $C_{ij}^k = <[e_i,e_j],e_k>$, where the Lie bracket $[\cdot,\cdot]$ is the usual matrix commutator. It is straightforward to compute
\begin{align}
    \label{eq:structure_constants}
    C_{ij}^k &= \frac{1}{\sqrt{2}} \;\; \textrm{if} \;\;ijk \;\;\textrm{is a direct cycle of}\;\; \{1,2,3\};\\
    C_{15}^6 &= - C_{16}^5 = - \sqrt{\beta} C_{24}^6 = \frac{1}{\sqrt{\beta}} C_{26}^4 = \sqrt{\beta} C_{34}^5 = -\frac{1}{\sqrt{\beta}} C_{35}^4 = \frac{1}{\sqrt{2}} .
\end{align}
and all others that cannot be deduced by skew-symmetry of the bracket are equal to $0$. Extend the inner-produt defined in the Lie algebra by $G$ to a left-invariant metric $g$ on $SE(3)$. Let $\nabla$ be its associated Levi-Civita connection. It is also left-invariant, and it is sufficient to know its values on left-invariant vector fields at the identity (identified with tangent vectors at the identity). These are linked to the structure constants by
\begin{equation*}
    \nabla_{e_i}e_j = \frac{1}{2} \sum_k (C_{ij}^k - C_{jk}^i + C_{ki}^j)e_k,
\end{equation*}
and can thus be computed explicitly thanks to \eqref{eq:structure_constants}. Let $\Gamma_{ij}^k = <\nabla_{e_i}e_j,e_k>$ be the associated Christoffel symbols. Let $\tau = (\sqrt{\beta} + \frac{1}{\sqrt{\beta}})$. We obtain
\begin{align}
    \label{eq:christoffels}
    \Gamma_{ij}^k &= \frac{1}{2\sqrt{2}} \;\; \textrm{if} \;\;ijk \;\;\textrm{is a cycle of [1,2,3]} ,\\
    \Gamma_{15}^6 &= - \Gamma_{16}^5 = - \frac{2}{\tau} \Gamma_{24}^6 = \frac{2}{\tau} \Gamma_{26}^4 = \frac{2}{\tau} \Gamma_{34}^5 = -\frac{2}{\tau} \Gamma_{35}^4 = \frac{1}{\sqrt{2}},
\end{align}
and all other are null.
Finally, recall that the curvature tensor and its covariant derivative at the identity can be defined $\forall u,v,w,z\in \se(3)$ by
\begin{align*}
    R(u,v) &= \nabla_u \nabla_v - \nabla_v \nabla_u - \nabla_{[u,v]} \\
    (\nabla_u R)(v,w)z &= \nabla_u \big(R(v,w)z \big) - R\big(\nabla_u v,w \big)z -R\big( v,\nabla_u w \big)z - R\big( v,w \big)\nabla_u z.
\end{align*}

\subsection{Geodesics}
\subsubsection{General Case}
For the computation of the geodesics, let $(x,v) \in TSE(3)$ and $\gamma$ be the geodesic such that $\gamma(0)=x, \dot \gamma(0)=v$. At all times $t$, define the Eulerian velocity $X(t) = (L_{\gamma^{-1}(t)})_* \dot \gamma(t)~\in~\se(3)$ where $*$ refers to the push-forward and $L$ to the left-translation $L_g:h\mapsto gh$ of $SE(3)$. Define the co-adjoint action $ad^*$ associated to the metric:
\begin{equation}
    \label{eq:ad_star}
    \forall a,b,c \in \se(3), \qquad <[a,b],c>=<ad_a^*c,b>.
\end{equation}
It can be computed explicitly using the structure constants. Then we have the evolution equations:
\begin{align}
    \label{eq:E-L}
    \dot \gamma(t) &= (L_{\gamma(t)})_* X(t) ,\\
    \dot X(t) &= ad_{X(t)}^* X(t) \label{eq:E-P} ,
\end{align}
with initial conditions $\gamma(0)=(R_0,v_0)$ and $\dot \gamma(0)= (Q_0, u_0)$, so that $X(0) = (L_{x^{-1}})_*(Q, u_0) = (R_0^T Q_0, R_0^T u_0) = (\Omega_0, d_0)$.

\begin{figure}
    \centering
    \includegraphics[width=0.6\textwidth]{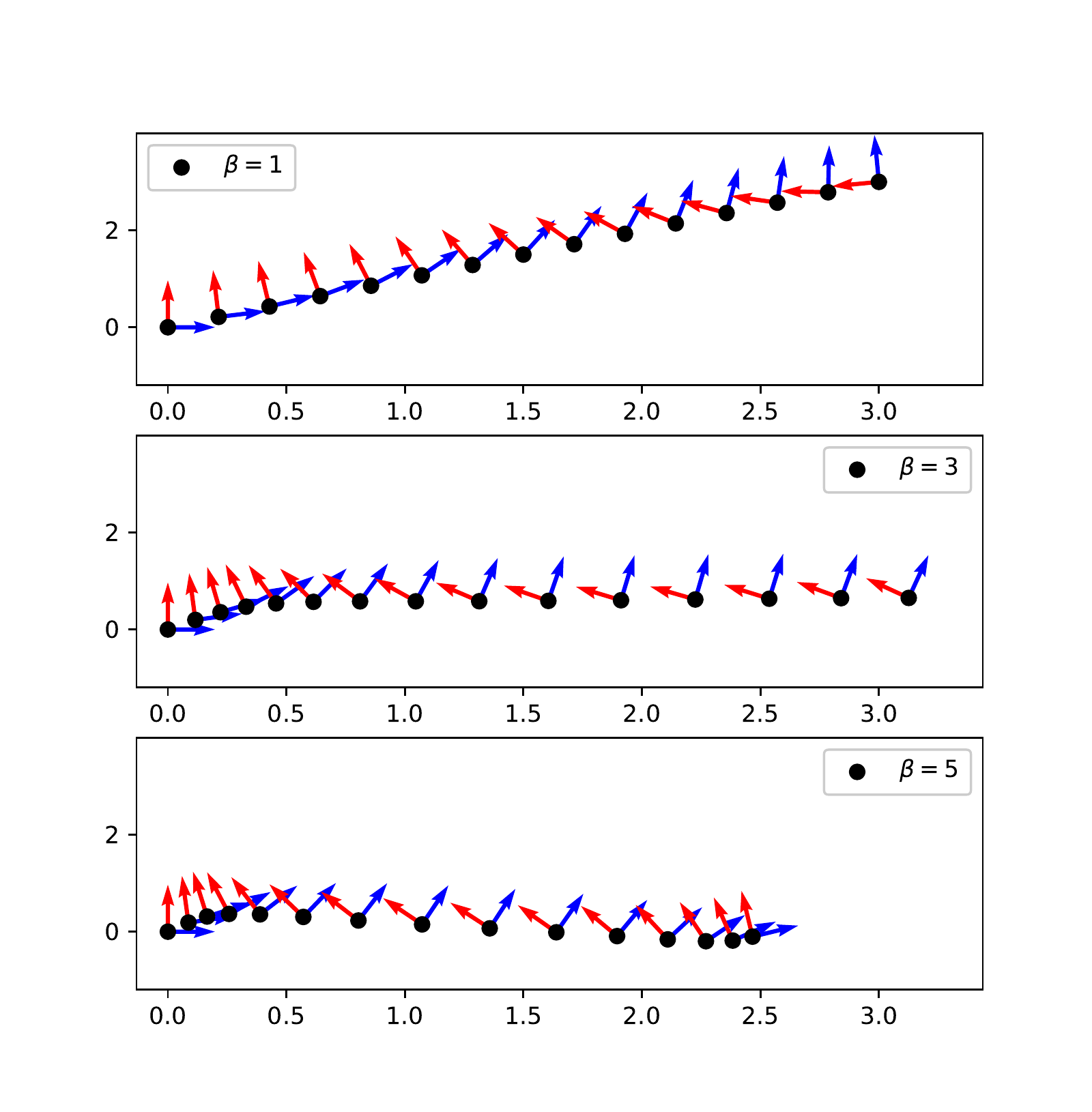}
    \vspace{-5mm}
    \caption{Visualization of geodesics of $SE(2)$ with the same initial conditions for different values of the anisotropy parameter $\beta$. The translation part are black dots, and the rotation part is applied to the orthonormal canonical frame of the 2d-plane.}
    \label{fig:geodesics_se2}
\end{figure}

\subsubsection{Case $\beta=1$}
Focusing on the case $\beta=1$ allows to validate our implementation by testing that we recover the direct product exponential map. This is proved in e.g. \cite{zefran_generation_1998}. It is straightforward to see that, in this case, $g$ coincides with the (direct) product of the canonical (bi-invariant) metrics $g_{rot}, g_{trans}$ of $SO(d)$ and $\R^d$.
So we have the general formula for geodesics from $(R, v) \in SE(3)$ with initial velocity $(\Omega, u) \in T_{(R,v)}SE(3)$: $\gamma(t) = (R \exp(tR^T\Omega), tu + v)$, and
\begin{align}
    \label{eq:exp}
    \exp_{(R,v)}(\Omega, u) = (R \exp(R^T\Omega), u + v),\\
    \log_{(R,v)}(Q, u) = (R\log(R^TQ), u - v). \label{eq:log}
\end{align}
These are in fact valid in $SE(d)$ for any $d \geq 2$. These geodesics are represented on Fig.~\ref{fig:geodesics_se2} for different values of $\beta$ in $SE(2)$ (where the metric matrix at identity is $G=diag(1,\beta, 1)$).

\subsection{Proof of lemma~\ref{lemma:sen}}
We now prove lemma~\ref{lemma:sen}, formulated as:
$(SE(3),g)$ is locally symmetric, i.e.\ $\nabla R=0$, if and only if $\beta = 1$. This is valid for any dimension $d\geq 2$ provided that the metric matrix $G$ is diagonal, of size $d(d+1)/2$, 
with ones everywhere except one coefficient of the translation part.

\begin{proof}
For $\beta=1$, $(SE(d),g)$ is isometric to $(SO(d)\times \R^d, g_{rot} \oplus g_{trans})$. As the product of two symmetric spaces is again symmetric, $(SE(d),g)$ is symmetric.

We prove the contraposition of the necessary condition. Let $\beta \neq 1$. We give $i,j,k,l$ s.t. $(\nabla_{e_i} R)(e_j,e_k)e_l \neq 0$:
\begin{align*}
    (\nabla_{e_3} R)(e_3,e_2)e_4 &= \nabla_{e_3} (R(e_3,e_2)e_4) - R(e_3,\nabla_{e_3} e_2)e_4 - R(e_3,e_2)\nabla_{e_3} e_4 \\
    &= \nabla_{e_3} (R(e_3,e_2)e_4) + \frac{1}{\sqrt{2}}R(e_3,e_1)e_4 - \frac{\tau}{2\sqrt{2}}R(e_3,e_2)e_5 .
\end{align*}
And from the above
\begin{align*}
    R(e_3,e_2)e_4 &= \nabla_{e_3} \nabla_{e_2}e_4 - \nabla_{e_2} \nabla_{e_3}e_4 - \nabla_{[e_3,e_2]}e_4 \\
        &= - \frac{\tau}{2\sqrt{2}} \nabla_{e_3}e_6 - \nabla_{e_2}e_5 + \frac{1}{\sqrt{2}}\nabla_{e_1}e_4\\
        &= 0.\\
    R(e_3,e_1)e_4 &= \nabla_{e_3} \nabla_{e_1}e_4 - \nabla_{e_1} \nabla_{e_3}e_4 - \nabla_{[e_3,e_1]}e_4 \\
        &= - \frac{\tau}{2\sqrt{2}} \nabla_{e_1}e_5 - \frac{1}{\sqrt{2}}\nabla_{e_2}e_4\\
        &= - \frac{\tau}{4} e_6 + \frac{\tau}{4} e_6 = 0.\\
    R(e_3,e_2)e_5 &= \nabla_{e_3} \nabla_{e_2}e_5 - \nabla_{e_2} \nabla_{e_3}e_5 - \nabla_{[e_3,e_2]}e_5 \\
        &= \frac{\tau}{2\sqrt{2}} \nabla_{e_2}e_4 + \frac{1}{\sqrt{2}}\nabla_{e_1}e_5\\
        &= -\frac{\tau^2}{8} e_6 + \frac{1}{2} e_6 = \frac{1}{2}(1 - \frac{\tau^2}{4})e_6 .\\
\end{align*}
And therefore
\begin{equation*}
    \beta \neq 1 \implies \tau = (\sqrt{\beta} + \frac{1}{\sqrt{\beta}}) \neq 2  \implies (\nabla_{e_3} R)(e_3,e_1)e_4 = - \frac{\tau}{4\sqrt{2}}(1 - \frac{\tau^2}{4})e_6 \neq 0,
\end{equation*}
which proves lemma~\ref{lemma:sen}.
\qed
\end{proof}

\section{Implementation}
Our implementation of the ladder methods, and the fanning scheme are made available online at \url{github.com/nguigs/ladder-methods}. The repository also contains a notebook to reproduce all the experiments of this paper. It relies on the open-source Python package \url{geomstats}, with its default \url{numpy} back-end. The automatic differentiation of the \url{autograd} package is used to compute the logs of the infinitesimal schemes, with \url{scipy}'s L-BFGS-B solver for the gradient descent.

\end{document}